\documentclass{article}
\usepackage[a4paper, total={6.5in, 10in}]{geometry}
\usepackage[utf8]{inputenc}
\usepackage[dvipsnames]{xcolor}
\usepackage{amsfonts}
\usepackage{amsthm}
\usepackage{amsmath}
\usepackage{parskip}
\usepackage{mathrsfs}
\usepackage{amssymb}
\usepackage{xtab}
\usepackage{setspace}
\usepackage{datetime}
\usepackage{svg}
\usepackage{csquotes}
\usepackage{fancyhdr}
\usepackage{float}
\usepackage{comment}
\usepackage{enumitem}
\usepackage{mathalfa}
\usepackage{array}
\usepackage{titlesec}
\usepackage{mdframed}
\usepackage{listings}
\usepackage{booktabs}
\usepackage{dsfont}
\usepackage{textcomp}
\usepackage{graphicx}
\usepackage{wrapfig}
\usepackage{cases}

\usepackage[font=small,labelfont=bf]{caption}

\usepackage[utf8]{inputenc}
\usepackage[english]{babel}

\usepackage{placeins}
\usepackage{subcaption}
\DeclareMathOperator*{\esssup}{ess\,sup}

\usepackage{hyperref}
\usepackage{cleveref}

\usepackage{cite}

\usepackage{multirow}

\graphicspath{{Figures/}}

\allowdisplaybreaks

\usepackage{thmtools}
\usepackage{thm-restate}
\declaretheorem[name=Proposition,numberwithin=section]{proposition}
\declaretheorem[name=Lemma,numberwithin=section]{lemma}
\declaretheorem[name=Theorem,numberwithin=section]{theorem}

\declaretheorem[name=Assumption]{assumption}

\newcounter{subassumption}

\usepackage{circuitikz}
\usetikzlibrary{decorations.pathreplacing,calligraphy}

\newcommand\vard{d}
\newcommand\varM{M}

\usepackage{lscape}
\usepackage{authblk}
\usepackage{tabularray}
\usepackage{natbib}

\title{\textbf{Opinion Dynamics with Continuous Age Structure}}
\author[1,2]{Andrew Nugent}
\author[2]{Susana N. Gomes}
\author[2]{Marie-Therese Wolfram}
\affil[1]{MathSys CDT, University of Warwick \protect\\ Email: \texttt{a.nugent@warwick.ac.uk}}
\affil[2]{Mathematics Institute, University of Warwick}

\nopagecolor

\begin{document}

\maketitle

\begin{abstract}
We extend a classical model of continuous opinion formation to explicitly include an age-structured population. We begin by considering a stochastic differential equation model which incorporates ageing dynamics and birth/death processes, in a bounded confidence type opinion formation model. We then derive and analyse the corresponding mean field partial differential equation and compare the complex dynamics on the microscopic and macroscopic levels using numerical simulations. We rigorously prove the existence of stationary states in the mean field model, but also demonstrate that these stationary states are not necessarily unique. Finally we establish connections between this and other existing models in various scenarios. 
\end{abstract}


\textbf{91D30} Social networks; opinion dynamics

\textbf{91C20} Clustering in the social and behavioral sciences

\textbf{Highlights:}
\renewcommand{\labelitemi}{{\tiny$\bullet$}}
\begin{itemize}
\itemsep-0.5em
    \item Introduce an individual-based model for opinion dynamics with an explicit continuous age structure.
    \item Derive the mean-field partial differential equation and provide examples of novel complex dynamics. 
    \item Rigorously prove the existence of steady states and discuss conditions for their uniqueness. 
    \item Establish connections with several existing models in various scenarios. 
\end{itemize}

\hrulefill

\section{Introduction} \label{Section: Introduction}

Opinion dynamics modelling aims to capture the mechanisms through which a population of individuals form and update their opinions, and answer questions about if, and on what scale, agreement emerges. The majority of models focus on consensus dynamics, in which interactions between individuals cause their opinions to move closer together. This raises the key question, initially posed by \cite{axelrod1997dissemination}, of why consensus is not more prevalent. 

One commonly accepted answer is that individuals have bounded confidence, meaning that they are only willing to interact with those who already share a sufficiently similar opinion. The models of Hegselmann and Krause \citep{hegselmann2002opinion} and Deffuant and Weisbuch \citep{deffuant2000mixing} both introduce this effect and the resulting clustering of opinions has been well studied in agent-based models \citep{lorenz2007continuous,hegselmann2015opinion,proskurnikov2018tutorial,blondel20072r}, ordinary and stochastic differential equation models for individuals opinions \citep{blondel2010continuous,motsch2014heterophilious,nugent2024bridging}, and partial differential equation models describing population-level opinion distributions \citep{garnier2017consensus,wang2017noisy,goddard2022noisy}. Furthermore, other contributing factors such as the presence of leaders or stubborn individuals \citep{during2015opinion,zhao2016bounded}, the impact of network structure \citep{amblard2004role,gabbay2007effects,kan2023adaptive,nugent2023evolving}, and the interactions between opinions on multiple topics \citep{jacobmeier2005multidimensional,fortunato2005vector,rodriguez2016collective} have been studied. 

An important effect is the addition of noise, with the simplest approach being to add a small amount of noise to each interaction, causing opinion diffusion \citep{garnier2017consensus,pineda2011diffusing}. Alternatively noise may be added to the opinions communicated between individuals or to the size of opinion updates \citep{steiglechner2024noise,nugent2024bridging}, or as `events' that may affect the entire population \citep{condie2021stochastic}. In 2009 \cite{pineda2009noisy} introduced `free will in the form of noisy perturbations', meaning that individuals opinions were occasionally randomly updated to a new randomly selected opinion. With this type of noise the opinion formation process continues to be dynamic, rather than remaining in the typically observed clusters. A similar approach has been taken by \cite{carletti2008birth} and \cite{grauwin2012opinion} in which this effect is described specifically as modelling the death/exit and birth/entry of individuals into the population. 

In this paper we model the effect of ageing more explicitly, defining an age for each individual and establishing at the microscopic (individual) level the mechanism through which individuals exit and re-enter the population with new opinions. This also allows us to incorporate the impact of individuals' ages on the way they interact with others. In addition, we look at the macroscopic (population) level by describing the evolution of the joint density over age and opinion. This approach is closer to the way that births and deaths are typically included in population dynamics and certain epidemiological models (see for example \cite{inaba2017age,de2023approximating,keyfitz1997mckendrick,perthame2006transport}) where individuals' ages are included explicitly in the model rather than being viewed solely as a source of noise. 

It is worth noting that this approach is very different from including a memory-dependent transition rate between opinion states, as in e.g. \cite{stark2008decelerating,llabres2024aging}, in which context an individuals `age' reflects the length of time they have held their current opinion rather than the biological age considered here. 

We begin by introducing the microscopic model in Section \ref{Section: Microscopic Model} and providing examples of interesting behaviours made possible by introducing continuous age structure. In Section \ref{Section: Macroscopic Model} we derive the macroscopic PDE model and provide examples of the corresponding behaviours. We then focus on the steady states of the macroscopic model in Section \ref{Section: steady states}, exploiting a connection with the classical mean-field model to show their existence, propose an efficient method to find them, and discuss when they are unique. In Section \ref{Section: PDE dynamics} we explore properties of the macroscopic system and in Section \ref{Section: Connection to other models} examine the connection with other existing models, before concluding and discussing future directions in Section \ref{Section: Conclusion}.

\section{Microscopic Model}  \label{Section: Microscopic Model}

Consider a finite population of $N$ individuals. The state of individual $i$ is described by the pair $(a_i,x_i)$ where $a_i\in[0,A)$ is the age of individual $i$ and $x_i\in U$ is their opinion. Throughout this paper we typically normalise to $A=1$ and consider $U = (-1,1)$ to represent the level of (dis)agreement with some statement, but the model could be rescaled for any bounded interval $U\subset\mathds{R}$ and $A>0$. The initial age of each individual is chosen uniformly at random in $[0,A)$ and the initial opinions are chosen randomly according to some given distribution $\rho_0$. These states evolve according to 
\begin{subequations} \label{Eqn: SDE model}
\begin{align}
    da_i &= \tau \, dt \,,\\
    dx_i &= \Bigg( \frac{1}{N} \sum_{j=1}^N \varM(a_i,a_j) \, \phi(x_j - x_i)\, (x_j - x_i) \Bigg) \, dt + \sigma \, d\beta_i & i=1,\dots,N\,,
\end{align}
\end{subequations}
except that when an individual reaches age $a_i = A$ their age is reset to $a_i = 0$ and a new opinion is chosen according to a given age-zero opinion distribution $\mu$. The interaction function $\phi:[-2,2]\rightarrow[0,1]$ is an odd function that describes the extent to which individuals interact based on the distance between their opinions. The age-interaction kernel $\varM:[0,A)^2 \rightarrow \mathds{R}_{\geq0}$ describes the strength/frequency of interactions between individuals of different ages, with the simplest case being $\varM\equiv1$ in which individuals' ages do not affect their interactions. The constant $\tau$ determines the timescale on which individuals age (relative to the change in their opinions). The constant $\sigma$ describes the amount of external noise affecting individuals opinions' and $\beta_i$ (for $i=1\dots,N$) are $N$ independent standard Brownian motions, modelling external influences on individuals' opinions. In general we are concerned with relatively small $\sigma$, so that the opinion dynamics is mostly driven by opinion interactions. We impose reflecting boundary conditions at the boundary of $U$, here $x=\pm1$. This model maintains a finite, fixed population size $N$.

\begin{assumption} \label{Assumption group: Standard}
    We make the following assumptions throughout:
    \begin{enumerate}[label={\textbf{\Alph*:}},ref={\theassumption\Alph*}]
        \item The interaction function $\phi \in C^2(\mathds{R})$. 
        \item The age-interaction kernel $\varM$ is Lipschitz continuous.
    \end{enumerate}
\end{assumption}

Under these Assumptions the system \eqref{Eqn: SDE model} is well-posed \citep{ikeda2014stochastic}. The age jump times are deterministic and can be determined from the initial conditions. As initial ages are selected uniformly at random in $[0,A)$ the probability of two individuals having the same age, and therefore jumping (resetting their opinions) at the same time, is zero. Between each jump the SDE system is well-posed as the vector field is Lipschitz continuous. Therefore a solution can be constructed by splitting the time interval according to the known jump times and piecing together the solution in each time interval. 

Note that due to the random replacement of individuals' opinions when they reach age $A$, this system cannot exhibit stationary states on the microscopic level. This provides one motivation for studying the mean-field model which, as will we see in Section \ref{Section: steady states}, exhibits possibly non-unique stationary states. 

This model could be altered by introducing stochastic death timing, meaning that each agent dies according to a Poisson process with a prescribed age-dependent death rate $d(a)$. This would allow a more realistic age distribution in the population. However, as is discussed in greater detail in Section \ref{Section: MK ageing}, a similar effect can be achieved by multiplying the age-interaction kernel $\varM(a,b)$ by the population-level age distribution resulting from this death rate. In fact, in the large-population limit these approaches are equivalent, provided that the death rate $d(a)$ ensures bounded ages, which is a realistic assumption in the context of opinion dynamics. 

In Figure \ref{fig:SDE examples} we provide a small number of examples that highlight several interesting features of the model \eqref{Eqn: SDE model}. These examples all take $\varM\equiv1$ and use a population of size $N=500$. The interaction function $\phi$ as well as $\sigma$ and $\tau$ are varied to show different behaviours, with the simulation length $T$ changed accordingly. We use a mollified version of the classical bounded confidence interaction function \citep{hegselmann2002opinion} so that $\phi$ satisfies Assumption \ref{Assumption group: Standard}. Specifying $r_2 > r_1 > 0 $ we take
\begin{align} \label{eqn: phi for examples}
    \phi(r) = 
    \begin{cases}
        1 & \text{ if } |r| < r_1 \,,\\
        \Tilde{\phi}(r) & \text{ if } r_1 \leq |r| \leq r_2 \,,\\
        0 & \text{ if } |r| > r_2 \,,
    \end{cases}
\end{align}
where $\Tilde{\phi}$ is a function that smoothly interpolates between $0$ and $1$.  

\begin{figure}
    \centering
    \begin{subfigure}{0.49\linewidth}
        \centering
        \includegraphics[width = \linewidth]{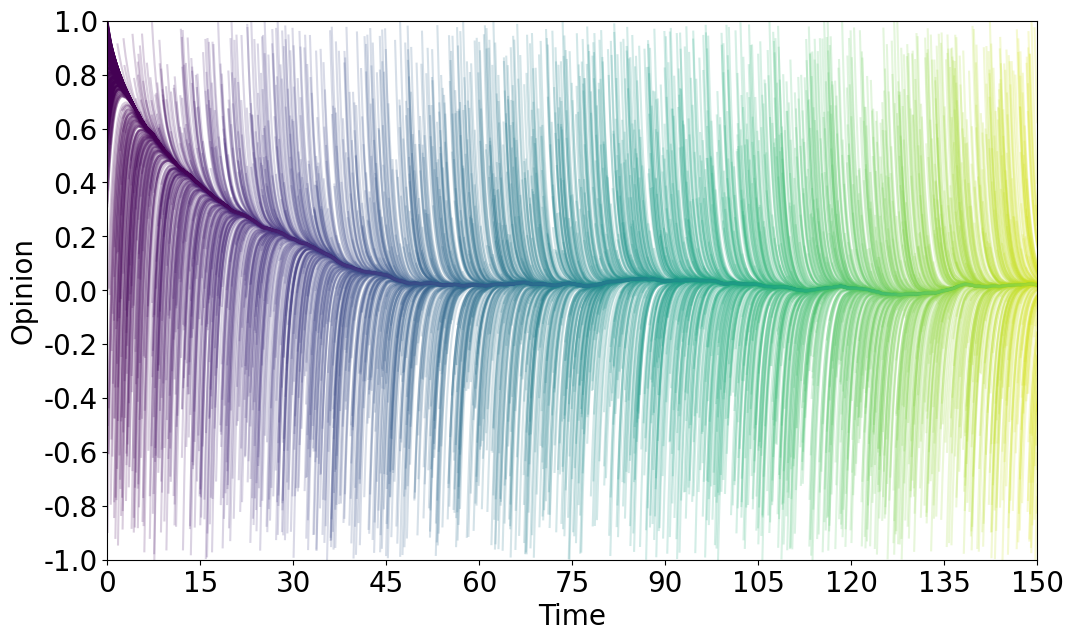}
        \caption{$\phi \equiv 1$, $\tau = 0.05$, $\sigma = 0.05$, $T = 150$.}
        \label{fig:SDE moving consensus}
    \end{subfigure}
    \begin{subfigure}{0.49\linewidth}
        \centering
        \includegraphics[width = \linewidth]{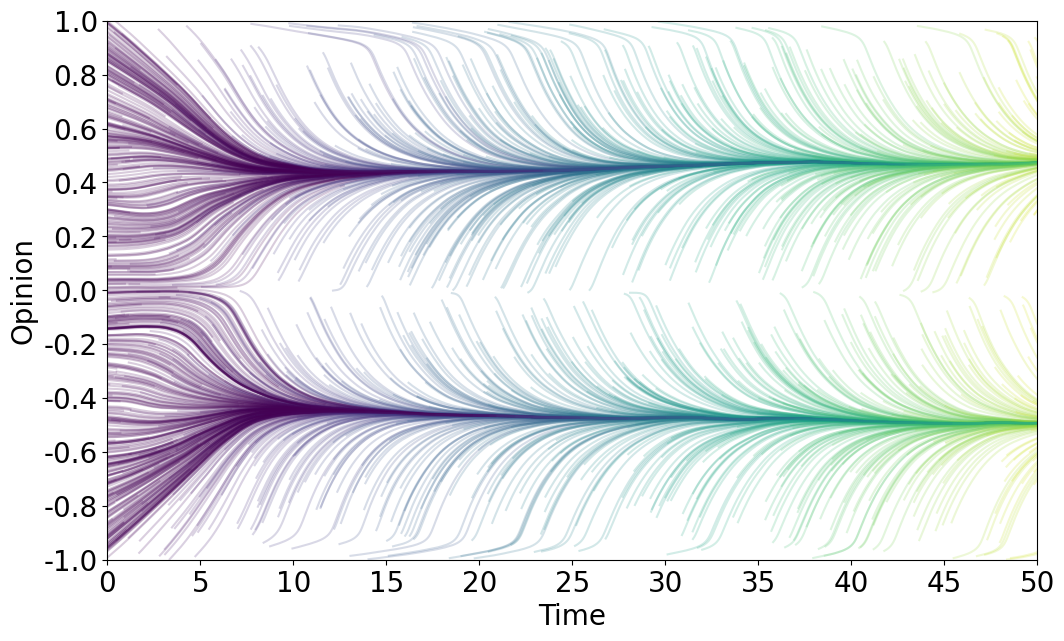}
        \caption{$r_1 = 0.4, r_2 = 0.5$, $\tau = 0.05$, $\sigma = 0.015$, $T = 50$.}
        \label{fig:SDE two clusters}
    \end{subfigure}
    \begin{subfigure}{0.49\linewidth}
        \centering
        \includegraphics[width = \linewidth]{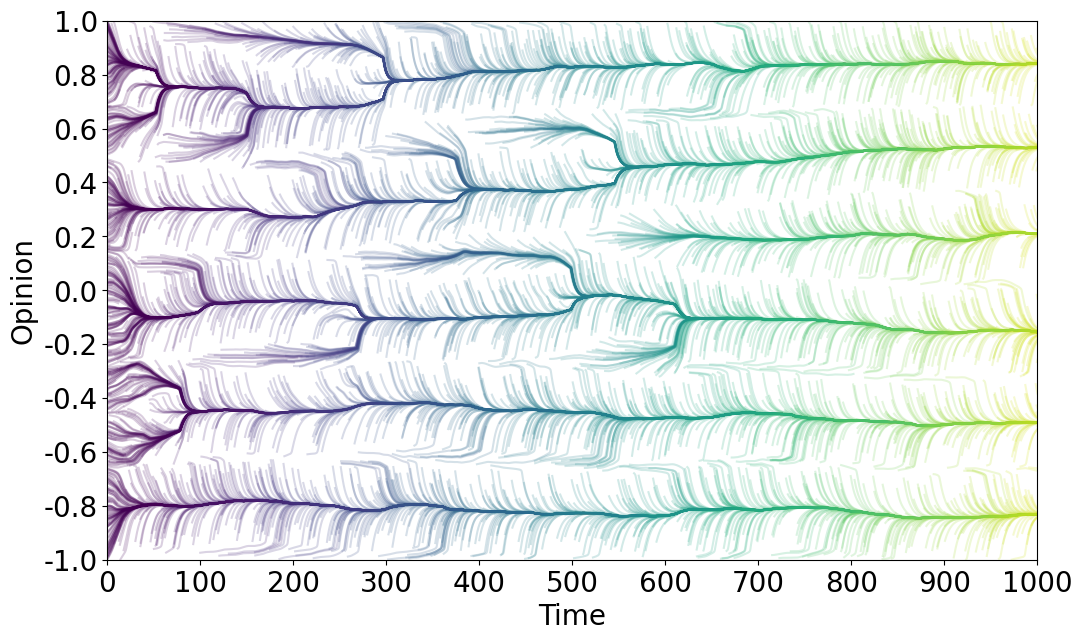}
        \caption{$r_1 = 0.14, r_2 = 0.16$, $\tau = 0.01$, $\sigma = 0.001$, $T = 1000$.}
        \label{fig:merging and emerging clusters}
    \end{subfigure}
    \begin{subfigure}{0.49\linewidth}
        \centering
        \includegraphics[width = \linewidth]{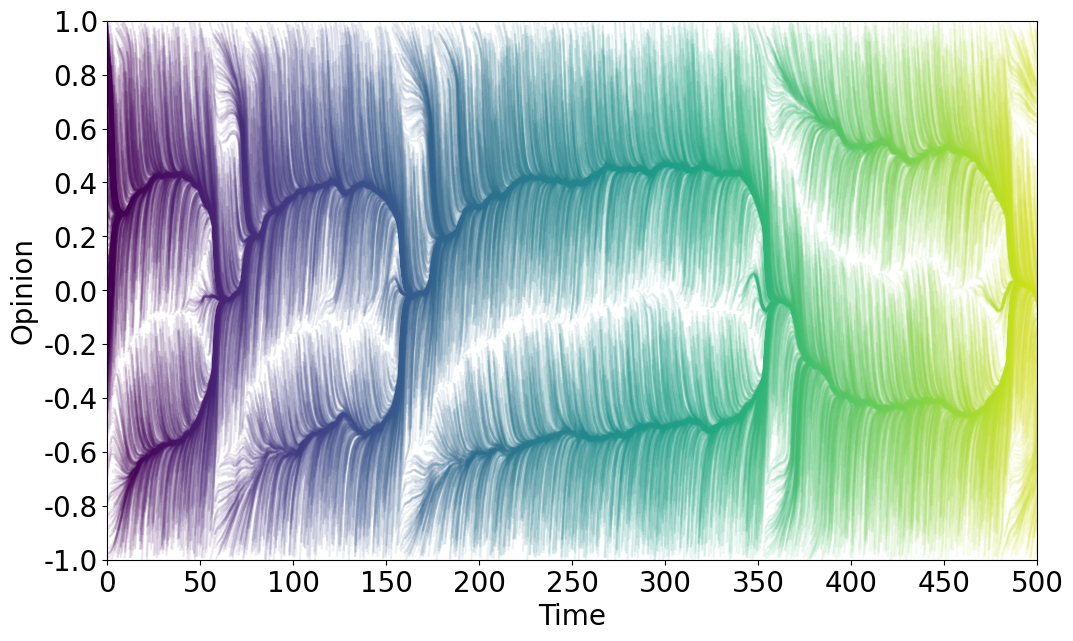}
        \caption{$r_1 = 0.48, r_2 = 0.58$, $\tau = 0.1$, $\sigma = 0.015$, $T = 500$.}
        \label{fig:SDE wiggles}
    \end{subfigure}
    \caption{Example numerical solutions to \eqref{Eqn: SDE model} using an Euler-Maruyama scheme with timestep $\Delta t = 0.01$. Individuals' opinion trajectories are plotted over time, with colour indicating the time at which individuals joined the population, with initially present individuals shown in purple. All examples use an interaction function $\phi$ given by \eqref{eqn: phi for examples}, $\varM \equiv 1$ and a uniform $\mu$. Figure \ref{fig:SDE two clusters} and Figure \ref{fig:merging and emerging clusters} also use a uniform $\rho_0$, while Figure \ref{fig:SDE moving consensus} uses an exponentially distributed $\rho_0$ \eqref{Eqn: positive skewed rho0} and Figure \ref{fig:SDE wiggles} uses a bimodal $\rho_0$ with peaks at $x=0$ and $x=-0.8$ \eqref{Eqn: bimodal rho0}. The values of parameters $r_1, r_2$, $\tau$ and $\sigma$ are varied to demonstrate different behaviours, with the final time $T$ changed accordingly.}
    \label{fig:SDE examples}
\end{figure}

Figures \ref{fig:SDE moving consensus} shows an example with $\phi\equiv1$, meaning all individuals interact with each other at all timepoints. The initial opinion distribution is given by 
\begin{align} \label{Eqn: positive skewed rho0}
    \rho_0(a,x) =  \kappa\, e^{-10\, (1-x)^2} \,,
\end{align}
where $\kappa$ is a normalisation constant. This means the initial distribution is heavily weighted towards positive opinions and indeed we observe in Figure \ref{fig:SDE moving consensus} that initially a consensus forms at a positive opinion. However, as new individuals enter the population with uniformly distributed opinions and quickly join this consensus, it slowly shifts to lie near $x=0$ with only minor deviations due to stochasticity.

In Figure \ref{fig:SDE two clusters} we consider a relatively large confidence bound with $r_2 = 0.5$ and a uniform $\rho_0$. As would be expected in the model without ageing this causes the population to split into two clusters. The bounded confidence radius of these two clusters partitions the opinion space, meaning each individual joining the population simply joins the nearest cluster and these are therefore stable over time. This indicates that, even though individuals may die there is the possibility of observing more structured macroscopically stable patterns. 

In Figure \ref{fig:merging and emerging clusters} the confidence bound is reduced to $r_2 = 0.16$ (again with a uniform $\rho_0$) and we observe many more smaller clusters. As these clusters are smaller and closer together there is significantly more merging of clusters due to stochasticity and individuals joining the population acting as `bridges' between nearby clusters, as seen in \cite{cahill2024modified}. In addition, new individuals entering the population provides the opportunity for new clusters to emerge, often at more extreme opinions or in the vacuum created following cluster merges. Towards the end of the simulation the population again settles into evenly spaced clusters. 

In contrast, Figure \ref{fig:SDE wiggles} shows a population that does not settle towards a macroscopically stationary pattern. Again we begin with a non-uniform initial opinion distribution, given by 
\begin{align} \label{Eqn: bimodal rho0}
    \rho_0(a,x) =  \kappa\, \big( e^{-10\, (0.8-x)^2} + e^{-10x^2} \big) \,,
\end{align}
where $\kappa$ is a normalisation constant. In this case the specific parameters chosen lead to the repeated formation of two clusters that then combine in a brief but unstable consensus before reforming again. The interaction function is similar to that in Figure \ref{fig:SDE two clusters} but the confidence bound has been increased slightly to $r_2 = 0.58, r_1 = 0.48$, allowing some interaction between the two clusters. Eventually they interact sufficiently to combine. However, the bounded confidence radius of the resulting consensus does not cover the whole opinion space (as it does in Figure \ref{fig:SDE moving consensus}) hence there is a `vacuum' in which new clusters form (as observed in Figure \ref{fig:merging and emerging clusters}). Therefore two new clusters emerge and the cycle begins again. Due to the stochasticity and random initial opinions of new individuals the opinion distribution is typically not symmetric, the consensus may form off-centre or one new cluster may merge into it, but the overall pattern appears to be stable. When considering the macroscopic PDE model this will translate into a predictable, periodic pattern. 

While the inclusion of births/deaths in the SDE model \eqref{Eqn: SDE model} creates various interesting behaviours it also creates new challenges in its analysis as the process no longer reaches a stationary distribution on the individual level. However, as we have observed in Figure \ref{fig:SDE examples} stability can emerge on the macroscopic level, hence we next consider a mean-field PDE model obtained in the large-population limit. 

\section{Macroscopic Model}  \label{Section: Macroscopic Model}

Motivated by the emergence of macroscopic behaviours observed in Section \ref{Section: Microscopic Model}, we would now like to obtain a mean-field limit of the form $\rho(t,a,x)$, which describes the density of individuals with age $a$ and opinion $x$ at time t. In particular we will be interested in the steady state behaviour of this model and the extent to which it mirrors the dynamics of the microscopic system \eqref{Eqn: SDE model}. 

We first set some notation to be used throughout this section. Define
\begin{equation} \label{Eqn: Definition of varphi}
    \varphi(y-x):=\phi(y-x)\,(y-x) \,.
\end{equation}
Note that since $\phi$ is odd the same is true of $\varphi$. In addition, $\varphi$ is also twice continuously differentiable. 

Recall that $A>0$ is the age at which individuals have their opinions reset, $U=(-1,1)$ is the opinion space and denote by $U_A = (0,A) \times U$ the joint age-opinion space. 

\subsection{Formal Derivation of mean-field limit} \label{Section: Formal Derivation}

We begin by considering the model without noise (i.e. $\sigma = 0$) and deriving a PDE satisfied by the empirical density, defined by
\begin{equation}
    \rho(t,a,x) = \frac{1}{N} \sum_{i=1}^N \delta_{a_i(t)}(a) \, \delta_{x_i(t)}(x) \,,
\end{equation}
where $\delta_{z}$ denotes a Dirac delta centred at $z$. Let $\zeta(a,x)$ be an infinitely differentiable, compactly supported test function. We formally compute 
\begin{align*}
    \int_U & \int_0^A \partial_t\rho(t,a,x) \,\zeta(a,x) \, da \, dx \\
    &= \frac{d}{dt} \int_U \int_0^A \frac{1}{N} \sum_{i=1}^N \delta_{a_i(t)}(a) \, \delta_{x_i(t)}(x) \, \zeta(a,x) \, da \, dx \\
    &= \frac{d}{dt} \frac{1}{N} \sum_{i=1}^N \zeta\big(a_i(t), x_i(t)\big) \\
    &= \frac{1}{N} \sum_{i=1}^N \partial_a \zeta\big(a_i(t), x_i(t)\big) \,\tau  + \frac{1}{N} \sum_{i=1}^N \partial_x \zeta\big(a_i(t), x_i(t)\big) \Bigg( \frac{1}{N} \sum_{j=1}^N \varM(a_i,a_j) \, \varphi(x_j - x_i) \Bigg) \\
    &= \int_U \int_0^A \frac{1}{N} \sum_{i=1}^N \delta_{a_i(t)}(a) \, \delta_{x_i(t)}(x) \Bigg( \partial_a \zeta(a, x) \,\tau + \partial_x \zeta(a, x) \, \frac{1}{N} \sum_{j=1}^N  \varM(a,a_j) \, \varphi(x_j - x) \Bigg) \, da \, dx \\
    &= \int_U \int_0^A \rho(t,a,x) \, \partial_a \zeta(a, x) \,\tau + \partial_x \zeta(a, x) \, \rho(t,a,x) \bigg( \int_U \int_0^A \varM(a,b) \, \varphi(y - x) \, \rho(t,b,y) \, db \, dy \, \bigg) \, da \, dx \\
    &= - \int_U \int_0^A \zeta(a, x) \, \partial_a\rho(t,a,x) \,\tau + \zeta(a, x) \, \partial_x \Bigg( \rho(t,a,x) \bigg( \int_U \int_0^A \varM(a,b) \, \varphi(y - x) \, \rho(t,b,y) \, db \, dy \, \bigg) \Bigg) \, da \, dx 
\end{align*}
Note that in the final line the integration by parts produces no boundary terms since $\zeta$ is compactly supported. This gives that $\rho$ is a weak solution to
\begin{equation}
    \partial_t\rho(t,a,x) + \tau \partial_a\rho(t,a,x) + \partial_x \Bigg( \rho(t,a,x) \bigg( \int_U \int_0^A \varM(a,b) \, \varphi(y - x) \, \rho(t,b,y) \, db \, dy \, \bigg) \Bigg) = 0 \,.
\end{equation}
When considering the model with noise the derivation is more complex as the empirical density no longer satisfies the limiting PDE. Indeed the stochastic effects only disappear in the large-population limit \citep{jabin2017mean}. In this case we note that the structure of \eqref{Eqn: SDE model} is similar to that of the second order Cucker-Smale model for bird flocking \citep{cucker2007emergent} and the general second order model discussed in \cite{jabin2017mean}, in which alignment and noise are only present in the evolution of velocities (here opinions). The derivation of the mean-field limit in this setting can be found in \cite{sznitman1991topics,graham1996asymptotic}. While we do not show this rigorously here, the similarity in structure implies that the corresponding mean-field limit for \eqref{Eqn: SDE model} is given by
\begin{equation} \label{Eqn: PDE with noise}
    \partial_t\rho + \tau \partial_a\rho + \partial_x \Bigg( \rho(t,a,x) \bigg( \int_U \int_0^A \varM(a,b) \, \varphi(y - x) \, \rho(t,b,y) \, db \, dy \, \bigg) \Bigg) - \, \frac{\sigma^2}{2} \partial_{x^2} \rho = 0 \,.
\end{equation}
Defining the flux, 
\begin{align}\label{Eqn: Flux}
    F[\rho](t,a,x) = \rho(t,a,x) \bigg( \int_U \int_0^A \varM(a,b) \, \varphi(y - x) \, \rho(t,b,y) \, db \, dy - \, \frac{\sigma^2}{2} \partial_{x}\log\big(\rho(t,a,x)\big) \bigg) \,,
\end{align}
then the PDE can be written as 
\begin{align*}
    \partial_t \rho + \tau \partial_a \rho + \partial_x F[\rho] = 0 \,.
\end{align*}
We also need to specify the initial and boundary conditions. We assume that the initial condition is a given probability distribution $\rho_0(a,x)$. As we imposed reflecting boundary conditions in the SDE model we impose no-flux boundary conditions on the boundary of $U$ (here $x=\pm1$) \citep{erban2007reactive,goddard2022noisy}. It is also necessary to specify the opinion distribution at age-zero. As with the microscopic system we assume that this follows a given distribution $\mu$. In order to maintain a constant population size we therefore set
\begin{align*}
    \rho(t,0,x) = \mu(x) \, \int_U \rho(t,1,y) \, dy \,,
\end{align*}
as the integral 
\begin{align*}
    \int_U \rho(t,1,y) \, dy \,,
\end{align*}
describes the total volume of individuals whose opinions are reset at time $t$. 

The complete initial-boundary value problem can be written 
\begin{subequations} \label{Eqn: PDE with ageing}
\begin{align} 
    \partial_t \rho(t,a,x) + \tau \partial_a \rho(t,a,x) +  \partial_x F[\rho](t,a,x) &= 0 \\[0.5em]
    F[\rho](t,a,-1) = F[\rho](t,a,1) &= 0 \\
    \rho(t,0,x) &= \mu(t,x) \, \int_U \rho(t,1,y) \, dy \label{Eqn: PDE with ageing age-zero BC} \\
    \rho(0,a,x) &= \rho_0(a,x)
\end{align}
\end{subequations}
Having derived this PDE problem we next look at several example solutions, before moving onto our analysis of its steady states. 

\subsection{Numerical simulations of the mean-field PDE} \label{Section: PDE examples}

In this Section we show numerical solutions of \eqref{Eqn: PDE with ageing}, some of which mirror the examples in Section \ref{Section: Microscopic Model}. A description of the numerical scheme used can be found in Appendix \ref{Appendix: Numerics}. 

In each of the examples shown in Figure \ref{fig:PDE examples} we take both the initial distribution $\rho_0(a,x)$ and the age-zero distribution $\mu$ to be uniform (with the exception of Figure \ref{fig:PDE wiggles}). We again use the mollified bounded confidence interaction function \eqref{eqn: phi for examples}.
By varying the confidence bounds, ageing rate and the strength of the noise/diffusion we observe different clustering behaviours. In each figure we show the evolution over time of the total opinion density,
\begin{align} \label{Eqn: Total opinion denstiy}
    P(t,x) = \int \rho(t,a,x) \, da \,,
\end{align}
for $x\in U,\, t \geq 0$. This allows a more direct comparison with the behaviours observed in the microscopic model. 

\begin{figure}[ht!]
    \centering
    \begin{subfigure}{0.49\linewidth}
        \centering
        \includegraphics[width = \linewidth]{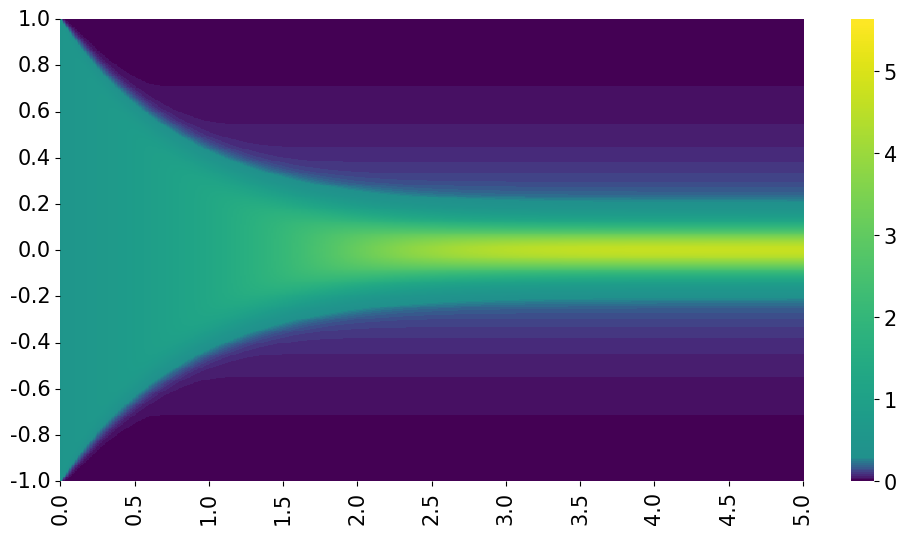}
        \caption{$\phi \equiv 1$, $\tau = 0.1$, $\sigma = 0.1$, $T = 5$.}
        \label{fig:PDE consensus}
    \end{subfigure}
    \begin{subfigure}{0.49\linewidth}
        \centering
        \includegraphics[width = \linewidth]{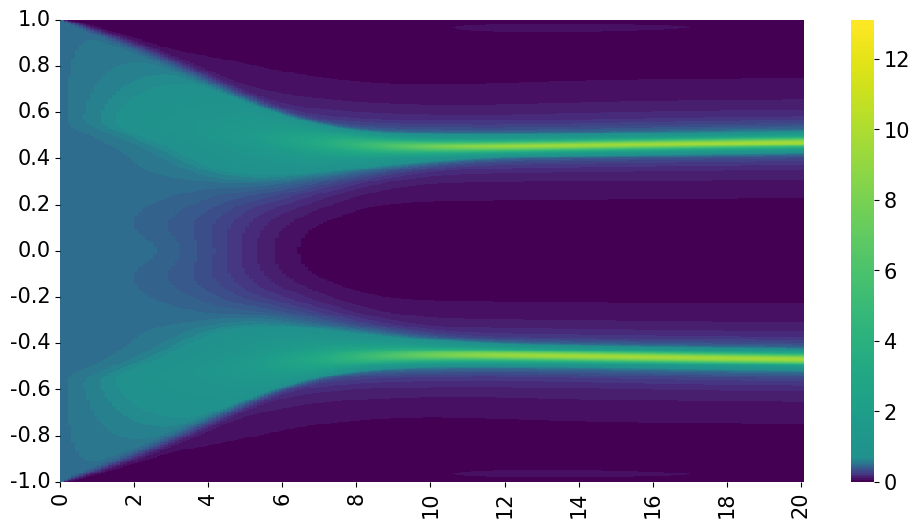}
        \caption{$r_1=0.4, r_2 = 0.5$, $\tau=0.05$, $\sigma=0.015$, $T=20$}
        \label{fig:PDE clusters}
    \end{subfigure}
    \begin{subfigure}{0.49\linewidth}
        \centering
        \includegraphics[width = \linewidth]{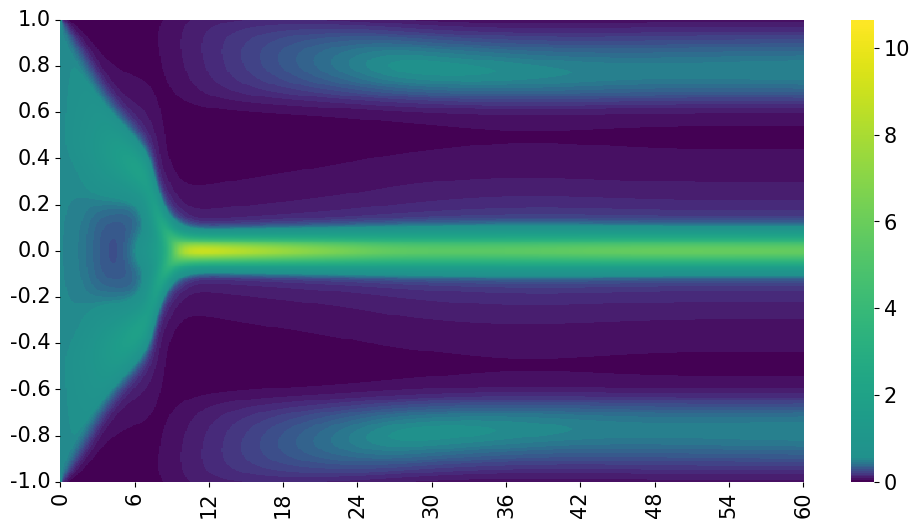}
        \caption{$r_1=0.5, r_2 = 0.6$, $\tau=0.05$, $\sigma=0.05$, $T=60$}
        \label{fig:PDE consensus with wings}
    \end{subfigure}
    \begin{subfigure}{0.49\linewidth}
        \centering
        \includegraphics[width = \linewidth]{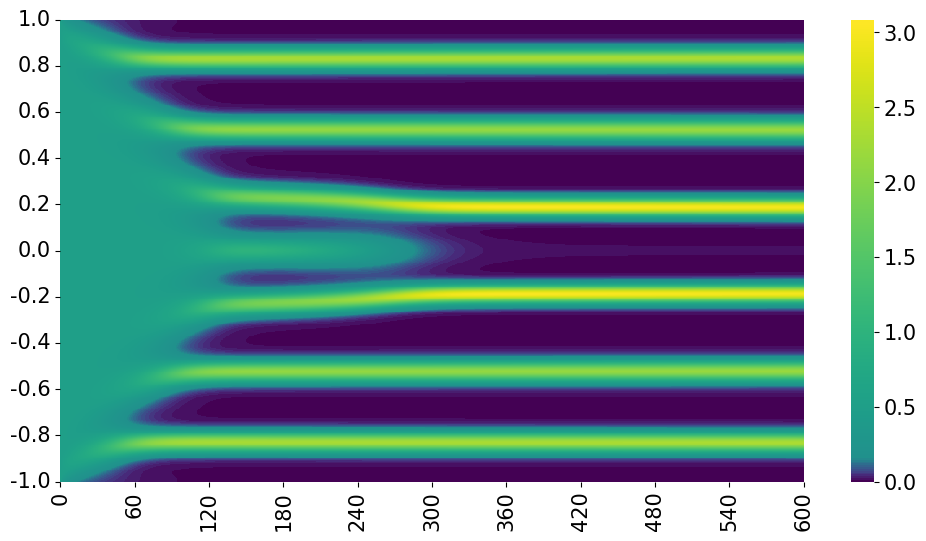}
        \caption{$r1 = 0.1, r2 = 0.11$, $\tau=0.001$, $\sigma=0.015$, $T=600$}
        \label{fig:PDE disappearing clusters}
    \end{subfigure}
    \begin{subfigure}{0.49\linewidth}
        \centering
        \includegraphics[width = \linewidth]{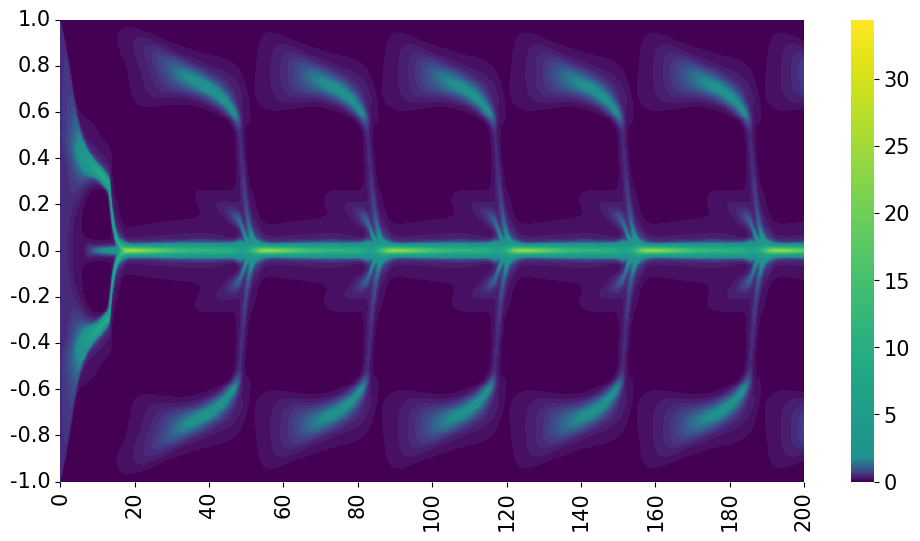}
        \caption{$r_1=0.48, r_2 = 0.58$, $\tau=0.05$, $\sigma=0.015$, $T=200$}
        \label{fig:PDE periodic consensus}
    \end{subfigure}
    \begin{subfigure}{0.49\linewidth}
        \centering
        \includegraphics[width = \linewidth]{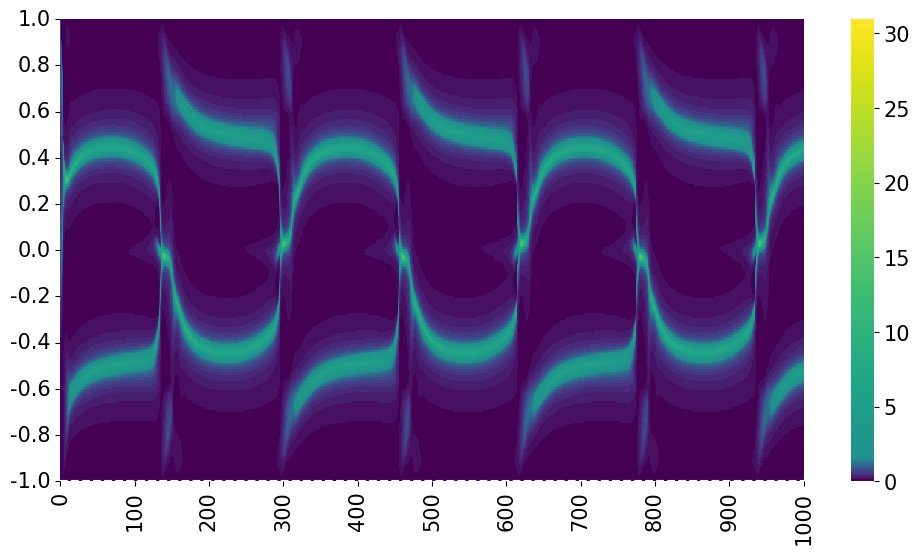}
        \caption{$r_1=0.48, r_2 = 0.58$, $\tau=0.1$, $\sigma=0.015$, $T=1000$}
        \label{fig:PDE wiggles}
    \end{subfigure}
    \caption{Example numerical solutions to \eqref{Eqn: PDE with ageing}. Colour indicates the total opinion density $P(t,x)$ \eqref{Eqn: Total opinion denstiy}. The colour map is shifted to make areas of low density more clearly visible. All examples use an interaction function $\phi$ given by \eqref{eqn: phi for examples}, $\varM \equiv 1$ and a uniform $\mu$ and $\rho_0$, with the exception of Figure \ref{fig:PDE wiggles} which uses \eqref{Eqn: bimodal rho0} for $\rho_0$. The values of parameters $r_1, r_2, \tau$ and $\sigma$ are varied to demonstrate different behaviours, with the final time $T$ changed accordingly. All solutions use the numerical scheme described in Appendix \ref{Appendix: Numerics} with $J_x = J_a = 500$ and $\Delta t = 0.001$.}
    \label{fig:PDE examples}
\end{figure}

Figure \ref{fig:PDE consensus} and Figure \ref{fig:PDE clusters} show examples in which the population quickly moves towards one and two stable clusters respectively, with Figure \ref{fig:PDE clusters} mirroring the behaviour observed in the microscopic system in Figure \ref{fig:SDE two clusters}. Here the ageing rate $\tau$ and level of noise $\sigma$ have a relatively minor effect compared to the strong clustering. Reducing either $\tau$ or $\sigma$ leads to more strongly peaked clusters, as individuals either have longer lifespans in which to concentrate their opinions or diffuse less away from the centre of the cluster. These effects can be seen by comparing Figure \ref{fig:PDE consensus} and Figure \ref{fig:PDE clusters}, in the former $\tau$ and $\sigma$ are relatively larger and the cluster(s) formed are therefore broader. In Figure \ref{fig:PDE consensus with wings} we show an intermediate point between the formation of one and two clear clusters. A strong central cluster forms, initially bringing the whole population to consensus. However, as new individuals enter the population, those with more extreme opinions are not brought into this consensus. This leads to the formation of two smaller, less concentrated, clusters at the extremes. 

In Figure \ref{fig:PDE disappearing clusters} we consider a smaller confidence bound. Initially the population remains near the uniform distribution due to the relatively strong diffusion, except at the boundary of the opinion space. As individuals at the boundary can only interact with those with less extreme opinions, small clusters form near the boundary, causing a subsequent decrease in density on the opposite side of the cluster. This in turn prompts the formation of another cluster closer to the centre, repeating until these clusters reach the centre of the opinion space. We also observe that by time $t = 400$ this central cluster cannot be maintained and all individuals in this cluster either die or join the clusters on either side. As in Figure \ref{fig:PDE clusters} the population reaches a steady state with almost evenly spaced, although not evenly sized, clusters. 

Figure \ref{fig:PDE periodic consensus} and Figure \ref{fig:PDE wiggles} show instances of periodic behaviour, similar to that seen in the microscopic model in Figure \ref{fig:SDE wiggles}. In Figure \ref{fig:PDE periodic consensus}, as in \ref{fig:PDE consensus with wings}, a strong central cluster quickly forms leaving a vacuum at the extreme opinions in which new, smaller clusters form. In this case these eventually combine with the central cluster before the cycle repeats again. This effect is only possible due to the influx of new individuals to the population, as without it the initial consensus would simply persist. The parameter range for which this behaviour occurs is relatively small, as it lies on the boundary between the population forming into two clusters (as in Figure \ref{fig:PDE clusters}) and the population forming one large central cluster and two smaller clusters that do not merge (as in Figure \ref{fig:PDE consensus with wings}). 

In Figure \ref{fig:PDE wiggles} we begin with a non-uniform initial condition given by \eqref{Eqn: bimodal rho0}, which has two peaks centred at $0$ and $-0.8$. Due to the uniform influx of new individuals these peaks move to become more symmetric. As in Figure \ref{fig:PDE periodic consensus} these clusters periodically combine and reform. However, in this case a lack of symmetry persists, meaning that after the two clusters combine only one new extreme cluster forms (rather than one forming near each boundary). This leads to a periodic shifting between positive and negative clusters that form a brief off-centre consensus that cannot be maintained. This directly mirrors the behaviour observed on the individual level in Figure \ref{fig:SDE wiggles}.

These examples demonstrate the rich variety of complex behaviours that are made possible by the inclusion of a continuous age-structure. In the following Section \ref{Section: steady states} we move to studying the possible stationary states of this model. 

\section{Stationary States} \label{Section: steady states}

To address the question of macroscopic stability in the SDE model \eqref{Eqn: SDE model}, as observed for example in Figure \ref{fig:SDE two clusters}, and the long-term behaviour seen in Figures \ref{fig:PDE clusters} and \ref{fig:PDE disappearing clusters}, we hope to establish the existence of steady states of the PDE model \eqref{Eqn: PDE with ageing}. That is, we are looking for a probability density $\rho(a,x)$ such that 
\begin{subequations} \label{Eqn: Steady state}
\begin{align}
    \tau \partial_a\rho(a,x) + \partial_x \Bigg( \rho(a,x) \bigg( \int_U \int_0^A \varM(a,b) \, \varphi(y - x) \, \rho(b,y) \, db \, dy \, \bigg) \Bigg) - \, \frac{\sigma^2}{2} \partial_{x^2} \rho(a,x) &= 0 \,, \label{Eqn: Steady state PDE}\\
    \rho(0,x) - \mu(x) \, \int_U \rho(1,y) \, dy  &= 0 \,,\label{Eqn: Steady state BC}\\
    \rho(a,1) \bigg( \int_U \int_0^A \varM(a,b) \, \varphi(y - 1)\, \rho(b,y) \, db \, dy \, \bigg) - \, \frac{\sigma^2}{2} \partial_{x} \rho(a,1) &= 0 \,,\\
    \rho(a,-1) \bigg( \int_U \int_0^A \varM(a,b) \, \varphi(y + 1) \, \rho(b,y) \, db \, dy \, \bigg) - \, \frac{\sigma^2}{2} \partial_{x} \rho(a,-1) &= 0 \,.
\end{align}    
\end{subequations}
If we assume that initially there is a uniform density in age then this will be preserved, meaning \eqref{Eqn: Steady state BC} can be replaced with
\begin{equation}
    \rho(0,x) = \mu(x) \,.
\end{equation}
Note that \eqref{Eqn: Steady state PDE} is similar in structure to the classical mean-field limit, with the age variable $a$ acting as `time'. However, the integral over $db$ inside the flux would then correspond to a time evolution in which individuals interacted with the future and past opinions of the population. Thus the PDE \eqref{Eqn: Steady state} cannot be solved forwards in age/time like the classical mean-field limit. To overcome this problem we will instead construct a mapping whose fixed point corresponds to a solution of \eqref{Eqn: Steady state}. 

For certain parameter regimes it can also take some time for the model to converge towards a steady state in numerical simulations, or we observe seemingly periodic behaviour. To this end we aim to find a simpler method to identify steady states that avoids solving either \eqref{Eqn: PDE with noise} or \eqref{Eqn: Steady state} directly. In proving the validity of this method we also prove the existence of steady states. 

\subsection{Existence of steady states} \label{Section: Existence of steady states}

In this section we make the following assumptions:
\begin{assumption} \label{Assumption group: Fixed point} Assume the following:
    \begin{enumerate}[label={\textbf{\Alph*:}},ref={\theassumption\Alph*}]
        \item For all $a\in[0,A]$ \label{Assumption: Initial age profile uniform.}
        \begin{align*}
            \int_U \rho_0(a,x) \, dx = 1 \,.
        \end{align*} 
        \item The age-interaction kernel $\varM(a,b)$ depends only upon $b$. That is, we can consider it as a function $\varM:[0,A]\rightarrow\mathds{R}_{\geq0}$, $b\mapsto \varM(b)$. \label{Assumption: A(a,b)=A{A}(b)}
        \item \label{Assumption: Integral of M is 1} The integral 
        \begin{align}
            \int_0^A \varM(b) \, db = 1 \,.
        \end{align} 
    \end{enumerate}
\end{assumption}
Assumption \ref{Assumption: Initial age profile uniform.} means that the initial ($t=0$) age profile is uniform. As will be seen in Section \ref{Section: Age transport} the age profile of the population satisfies a transport equation, so Assumption \ref{Assumption: Initial age profile uniform.} ensures that the age profile remains uniform at all times. Assumption \ref{Assumption: A(a,b)=A{A}(b)} ensures that all individuals view the population-level opinion density in the same way, regardless of their own age. This assumption will be crucial in performing the analysis below. As discussed in Section \ref{Section: MK ageing}, even with these two assumptions the model can still effectively capture realistic ageing effects and non-uniform population level age distributions. Note that Assumption \ref{Assumption: Integral of M is 1} can be achieved by appropriately rescaling $A$, $\tau$ and $\sigma$ and is made to reduce the number of parameters. 

As the age-interaction kernel $\varM$ is Lipschitz continuous on a finite domain $U$, it is also bounded and we denote the maximum value of $\varM$ by 
\begin{align*}
    \varM_{\max} = \max\limits_{0\leq b \leq A} \varM(b) \,.
\end{align*}

We can now rewrite \eqref{Eqn: Steady state PDE} as 
\begin{align*}
    \tau \partial_a\rho(a,x) + \partial_x \Bigg( \rho(a,x) \bigg( \int_U \varphi(y - x) \bigg( \int_0^A \varM(b) \, \rho(b,y) \, db \bigg) \, dy \, \bigg) \Bigg) - \, \frac{\sigma^2}{2} \partial_{x^2} \rho(a,x) &= 0 \,.
\end{align*}
To show the existence of a steady state we will use a fixed point argument. To this end assume temporarily that the value of the following integral, which is constant in $a$, is known
\begin{align*}
    \lambda(y) = \int_0^A \varM(b) \, \rho(b,y) \, db \,.
\end{align*}
This integral describes the population-level opinion distribution with which individuals interact. From this, we define $\Lambda:U\rightarrow\mathds{R}$ by
\begin{align} \label{Eqn: Definition of Lambda}
    \Lambda(x) := \int_U \varphi(y - x) \, \lambda(y) \, dy \,,
\end{align}
and let $\Tilde{\rho}(a,x)$ be a solution to the initial-boundary value problem
\begin{subequations} \label{Eqn: lambda steady state}
\begin{align}
    \partial_a\rho(a,x) + \partial_x \Big( \rho(a,x) \Lambda(x) \Big) - \, \frac{\sigma^2}{2} \partial_{x^2} \rho(a,x) &= 0 \,, \label{Eqn: lambda steady state PDE}\\
    \rho(0,x) &= \mu(x) \,,\label{Eqn: lambda steady state BC}\\
    \rho(a,1) \, \Lambda(1) - \, \frac{\sigma^2}{2} \partial_{x} \rho(a,1) &= 0 \,,\\
    \rho(a,-1) \, \Lambda(-1) - \, \frac{\sigma^2}{2} \partial_{x} \rho(a,-1) &= 0 \,.
\end{align}    
\end{subequations}
From this solution, define
\begin{align}
    \Tilde{\lambda}(y) &= \int_0^A \varM(b) \, \Tilde{\rho}(b,y) \, db \,,  \label{Eqn: lambda consistency equation} 
\end{align}
and note that if $\Tilde{\lambda} = \lambda$ we have in fact recovered \eqref{Eqn: Steady state PDE}. That is, there is a one-to-one correspondence between stationary states of \eqref{Eqn: PDE with ageing} and fixed points of the mapping $\lambda \mapsto \Tilde{\lambda}$ described above. We now rigorously show the existence of such stationary states through the following steps:
\begin{itemize}
    \item[] \textbf{Step 1:} Determine conditions for the existence of solutions to \eqref{Eqn: lambda steady state} in an appropriate space. 
    \item[] \textbf{Step 2:} Introduce the mapping $\lambda \mapsto \Tilde{\lambda} =: \mathcal{F}(\lambda)$ and show that $\mathcal{F}:L^\infty(U)\rightarrow L^\infty(U)$ is well-defined.
    \item[] \textbf{Step 3:} Find an appropriate set $K\subset L^\infty(U)$ so that we may apply Schauder's Fixed Point Theorem to $\mathcal{F}:K\rightarrow K$.
\end{itemize}

\underline{\smash{\textbf{Step 1:}}} Determine conditions for the existence of solution to \eqref{Eqn: lambda steady state} in an appropriate space.

We will see that it is sufficient to consider weak solutions of \eqref{Eqn: lambda steady state}. We therefore begin by defining and setting the notation for weak derivates and several useful function spaces, following the definitions in \cite{evans2022partial} and \cite{ladyzhenskaia1968linear}. 

For $1\leq p \leq \infty$ let $L^p(U_A)$ be the Banach space of all measurable functions on $U_A$ with finite norm
\begin{align*}
    \| u \|_{L^p(U_A)} = 
    \begin{cases}
        \bigg( \int_0^A \int_U |u(a,x)|^p \, dx \, da \bigg)^{1/p} & p < \infty \,, \\
        \esssup_{(a,x)\in U_A} |u(a,x)| & p = \infty \,.
    \end{cases}
\end{align*}
We use the shorthand notation $\| u \|_p$ for this norm. 

Suppose $u \in L^1(U_A)$. We say that $v \in L^1(U_A)$ is the weak partial $x$-derivative of $u$, written $D_x u = v$ provided
\begin{align*}
    \int_0^A \int_U u(a,x) \, D_x \zeta(a,x) \, dx \, da = - \int_0^A \int_U v(a,x) \, \zeta(a,x) \, dx \, da
\end{align*}
for all test functions $\zeta:U_A\rightarrow\mathds{R}$ that are infinitely differentiable with compact support in $U_A$. The second weak $x$-derivative and the weak $a$-derivative, denoted $D^2_x$ and $D_a$ respectively, are defined similarly. 

Denote by $W_2^{2,1}(U_A)$ the Sobolev space consisting of the elements of $L^2(U_A)$ having weak derivatives of the form $D_a^r D_x^s$ for any $r$ and $s$ with $2r + s \leq 2$. The norm in this space is 
\begin{align*}
    \| u \|_{W_2^{2,1}(U_A)} = \bigg( \| u \|_2^2 + \| D_x u \|_2^2 + \| D_x^2 u \|_2^2 + \| D_a u \|_2^2 \bigg)^{1/2} \,.
\end{align*}
While $u \in W_2^{2,1}(U_A)$ are functions of both age and opinion, it will at times be convenient to think of them instead as mapping ages $(0,A)$ into a space of functions over opinions. Denote by $H^1(U)$ the Sobolev space consisting of the elements of $L^2(U)$ with weak $x$-derivative and norm
\begin{align*}
    \| \nu \|_{H^1(U)} = \bigg( \| \nu \|_{L^2(U)}^2 + \| D_x \nu \|_{L^2(U)}^2 \bigg)^{1/2} \,.
\end{align*}
For any $u \in W_2^{2,1}(U_A)$, $u(a,\cdot)$ must lie in $H^1(U)$ for almost every $a\in(0,A)$. Thus we can associate $u$ with a function $\underline{u}:(0,A)\rightarrow H^1(U)$ given by $\underline{u}(a) = u(a,\cdot)$ (after possibly redefining $u$ on a set of measure zero). Such functions are elements of a Bochner space: for $1\leq p < \infty$ let $L^p(0,A;H^1(U))$ denote the Bochner space of functions $\underline{u}:(0,A) \rightarrow H^1(U)$ with finite norm
\begin{align*}
    \|\underline{u}\|_{L^p(0,A;H^1(U))} := \bigg( \int_0^A \|\underline{u}(a)\|_{H^1(U)}^p \, da \bigg)^{1/p} \,.
\end{align*}
By comparing the definitions of each norm, we see that for any $u \in W_2^{2,1}(U_A)$, $\underline{u} \in L^2(0,A;H^1(U))$ with 
\begin{align} \label{Eqn: bound on u in $L^2(0,T;H1(U))$}
    \|\underline{u}\|_{L^2(0,A;H^1(U))} \leq \|u\|_{W_2^{2,1}(U_A)} \,.
\end{align}
(See Lemma \ref{Lemma: underline well defined} in Appendix \ref{Appendix: Proofs} for details).

We aim to apply Theorem 9.1 from Chapter IV Section 9 of \cite{ladyzhenskaia1968linear} to guarantee the existence and uniqueness of weak solutions to \eqref{Eqn: lambda steady state}. Using the notation established above, we now clarify the notion of a weak solution to \eqref{Eqn: lambda steady state} as an element $\rho\in W^{2,1}_2(U_A)$.

Assume first that \eqref{Eqn: lambda steady state} has a classical solution $\rho(a,x)$. Take a test function $\zeta(a,x)$ that is twice continuously differentiable, multiply $\rho(a,x)$ by $\zeta(a,x)$ and integrate over $U_A$ to give 
\begin{align}
    0 &= \int_0^A \int_U \bigg( \tau\partial_a \rho + \partial_x ( \rho \Lambda ) - \frac{\sigma^2}{2} \partial_{x^2} \rho \bigg) \zeta \, dx \, da \nonumber\\
    &= \tau\int_0^A \int_U \partial_a \rho \, \zeta \, dx \, da + \int_0^A \int_U \partial_x \Big( \rho \Lambda - \frac{\sigma^2}{2} \partial_x \rho \Big) \, \zeta \, dx \, da \nonumber\\
    &= \tau\int_0^A \int_U \partial_a \rho \, \zeta \, dx \, da 
    + \int_0^A \Big[ \Big(\rho \Lambda - \frac{\sigma^2}{2}\partial_x \rho \Big)\zeta \Big]_{x=-1}^{x=1} \, da - \int_0^A \int_U \Big(\rho \Lambda - \frac{\sigma^2}{2}\partial_x \rho \Big) \, \partial_x\zeta \, dx \, da \label{Eqn: developing weak formulation}
\end{align}
For an element $\rho \in W^{2,1}_2(U_A)$ the weak partial derivatives $D_a \rho$ and $D_x \rho$ exist and, since $\rho$ is twice weakly differentiable in $x$ we may apply the trace theorem to $\rho$ and $D_x \rho$ to ensure their values at the boundary $x \in \{-1,1\}$ are well-defined. Thus we can prescribe the boundary condition 
\begin{align*}
    \rho \Lambda - \frac{\sigma^2}{2} D_x \rho = 0 \,,
\end{align*}
on $S_A = (0,A)\times\{-1,1\}$, hence
\begin{align*}
    \int_0^A \Big[ \Big(\rho \Lambda - \frac{\sigma^2}{2} D_x \rho \Big)\zeta \Big]_{x=-1}^{x=1} \, da = 0 \,.
\end{align*}
We therefore call $\rho \in W^{2,1}_2(U_A)$ a weak solution of \eqref{Eqn: lambda steady state} if the following holds for any test function $\zeta$: 
\begin{subequations} \label{Eqn: Weak formulation of lambda steady state eq}
    \begin{align}
        \int_0^A \int_U D_a \rho \, \zeta \, dx \, da  - \frac{1}{\tau} \int_0^A \int_U \Big(\rho \Lambda - \frac{\sigma^2}{2} D_x \rho \Big) \, D_x\zeta \, dx \, da &= 0 \,,\\[0.2em]
        \rho(a,1) \Lambda(a,1) - \frac{\sigma^2}{2} D_x \rho(a,1) &= 0  \,,\\[0.3em]
        \rho(a,-1) \Lambda(a,-1) - \frac{\sigma^2}{2} D_x \rho(a,-1) &= 0  \,,\\[0.3em]
        \rho(0,x) &= \mu(x) \,.
    \end{align}
\end{subequations}

In order to apply the desired existence and uniqueness Theorem we will require the following the following conditions::
\begin{enumerate}
    \item[(L1)] $\Lambda \in L^3(U_A)$.
    \item[(L2)] $\Lambda' \in L^2(U_A)$, where $'$ denotes a derivative.
    \item[(L3)] $\|\Lambda\|_{L_3(U_{t,t+s})} \rightarrow 0 $ as $s \rightarrow 0$. 
    \item[(L4)] $\|\Lambda'\|_{L^2(U_{t,t+s})} \rightarrow 0 $ as $s \rightarrow 0$.
    \item[(L5)] $\mu \in L^2(U)$\,.
\end{enumerate}

If we assume that $\lambda \in L^\infty(U)$ then we have immediately that $\Lambda \in L^\infty(U) \subset L^3(U)$, with the inclusion holding since $U$ is bounded. In addition as $\phi\in C^2(U)$ we may use that $\lambda \in L^\infty(U)$ to exchange the integration and differentiation below to obtain 
\begin{align}
    \Lambda'(x) 
    &= \frac{d}{dx} \int_U \varphi(y - x)\, \lambda(y) \, dy \,,\nonumber\\
    &= \int_U \frac{d}{dx} \varphi(y - x)\, \lambda(y) \, dy \,,\nonumber\\
    &= - \int_U \Big( \phi'(y - x)\,(y-x) + \phi(y-x) \Big)\, \lambda(y) \, dy \,, \label{Eqn: Definition of Lambda'}
\end{align}
hence $\Lambda'\in L^\infty(U) \subset L^2(U)$. 

In addition to the assumption that $\lambda \in L^\infty(U)$, we note that we intend to define $\mathcal{F}(\lambda)$ as the integral of $\rho$, which is itself a probability density. Hence we also assume that $\lambda$ is positive almost everywhere with
\begin{align*}
    \|\lambda\|_{L^1(U)} = \int_U \lambda(y) \, dy = 1 \,.
\end{align*}
To summarise, we assume that 
\begin{align*}
    \lambda \in K_1 := \{ \lambda \in L^\infty(U) : \lambda \geq 0 \text{ a.e.}, \| \lambda \|_{L^1(U)}=1 \} \,.
\end{align*}
From this we can now write the following Proposition \ref{Prop: Uniform bounds on Lambda}. 

\begin{proposition} \label{Prop: Uniform bounds on Lambda}
    The functions $\Lambda$ and $\Lambda'$, defined by \eqref{Eqn: Definition of Lambda} and \eqref{Eqn: Definition of Lambda'} respectively, are bounded uniformly with respect to $\lambda \in K_1$. That is, there exists a constant $C_{\phi'}$, depending on $\phi'$ and $U$ only, such that for all $\lambda \in K_1$ and all $x \in U$, 
    \begin{align*}
        |\Lambda(x)| \leq 2 \quad\text{and}\quad |\Lambda'(x)| \leq C_{\phi'} \,.
    \end{align*}
\end{proposition}
\begin{proof}
    For any fixed $x \in U$ we have
    \begin{align*}
        |\Lambda(x)| 
        &= \bigg| \int_U \varphi(y - x) \, \lambda(y) \, dy \bigg| \\
        &\leq \int_U \big| \varphi(y - x) \big| \, \big|\lambda(y)\big| \, dy \\
        &\leq 2 \int_U \, \big|\lambda(y)\big| \, dy \\
        &= 2 \,.
    \end{align*}
    Since $|\phi|\leq 1$ and $\phi'$ is continuous, and therefore bounded on $U$, there is also some constant $C_{\phi'}$, which depends on $\phi$ and $U$ only, such that \[|\phi'(y - x)\,(y-x) + \phi(y-x)| \leq C_{\phi'}\] for all $x,y$. Hence an almost identical argument to that above shows the second inequality. 
\end{proof}

Since both $\Lambda$ and $\Lambda'$ are constant in time, we have therefore satisfied conditions L1-L4 above. It will be necessary to later verify that the set $K_1$, or some subset of it, is invariant under $\mathcal{F}$ to ensure that these conditions are maintained. 

We can now apply Theorem 9.1 from Chapter IV Section 9 of \cite{ladyzhenskaia1968linear} (with $f\equiv0, \Phi \equiv 0$) to conclude the existence of a unique solution $\rho \in W_2^{2,1}(U_A)$ to \eqref{Eqn: lambda steady state} that additionally satisfies 
\begin{align}
    \|\rho\|_{W_2^{2,1}(U_A)} \leq c_1 \, \| \mu \|_{L^2(U)} \,, \label{Eqn: bound on rho in Sob space}
\end{align}
where $\mu \in L^2(U)$ is the opinion distribution specified at age zero. Note that the constant $c_1$ depends upon the final age $A$ but is bounded for any $A < \infty$. The constant $c_1$ also depends upon the coefficients of \eqref{Eqn: lambda steady state} but since, by Proposition \ref{Prop: Uniform bounds on Lambda}, these coefficients are bounded uniformly for $\lambda\in K_1$, we may also take a constant $c_1$ such that \eqref{Eqn: bound on rho in Sob space} holds for any $\lambda \in K_1$. 

\underline{\smash{\textbf{Step 2:}}} Introduce the mapping $\lambda \mapsto \Tilde{\lambda} =: \mathcal{F}(\lambda)$ and show that $\mathcal{F}:L^\infty(U)\rightarrow L^\infty(U)$ is well-defined.

Letting $\rho \in W_2^{2,1}(U_A)$ be the unique solution to \eqref{Eqn: lambda steady state} with a given $\lambda \in L^\infty(U)$, we define
\begin{align} \label{Eqn: Definition of F mapping}
    \mathcal{F}(\lambda) := \int_0^A \varM(b) \, \underline{\rho}(a) \, da \,,
\end{align}
where the integral is a Bochner integral (see for example \textsection 2.2 of \cite{diestel1977vector} for further details). As $\rho\in W_2^{2,1}(U_A)$, \eqref{Eqn: bound on u in $L^2(0,T;H1(U))$} gives that $\underline{\rho}$ is a well-defined element of $L^2(0,A;H^1(U))$. As $[0,A]$ is bounded, $L^2(0,A;H^1(U)) \subset L^1(0,A;H^1(U))$ and so H\"older's inequality gives 
\begin{align} \label{Eqn: bounding L1(0,T;H1) by L2(0,T;H1)}
    \| \underline{\rho} \|_{L^1(0,A;H^1(U))} \,\leq \sqrt{A} \| \underline{\rho} \|_{L^2(0,A;H^1(U))} \,.
\end{align}
For each age $a\in(0,A)$, $\varM(a)$ is simply a constant, thus
\begin{align*}
    \| \varM\underline{\rho} \|_{L^1(0,A;H^1(U))} =  \int_0^A \|\varM(a)\,\underline{\rho}(a)\|_{H^1(U)}^2 \, da =  \int_0^A \varM(a)\,\|\underline{\rho}(a)\|_{H^1(U)}^2 \, da \leq   \varM_{\max} \,\int_0^A \|\underline{\rho}(a)\|_{H^1(U)}^2 \, da \,.
\end{align*}
Hence $\varM\underline{\rho} \in L^1(0,A;H^1(U))$, so Theorem 2 in \textsection 2.2 of \cite{diestel1977vector} gives that the Bochner integral \eqref{Eqn: Definition of F mapping} is well-defined. 

\underline{\smash{\textbf{Step 3:}}} Find an appropriate set $K\subset L^\infty(U)$ so that we may apply Schauder's Fixed Point Theorem to $\mathcal{F}:K\rightarrow K$.

The above also shows that $\mathcal{F}(\lambda) \in H^1(U)$, although this alone does not guarantee that $\mathcal{F}(\lambda)$ remains in $L^\infty(U)$ as required. For this we apply the following Lemma \ref{Lemma: Bounding lambda in L infinity with the norm of rho in Sob}. 

\begin{lemma} \label{Lemma: Bounding lambda in L infinity with the norm of rho in Sob}
    There exists a constant $c_2$ such that for any $u \in W_2^{2,1}(U_A)$ the following estimate holds
    \begin{align}
        \bigg\| \int_0^A \underline{u}(a) \, da \, \bigg\|_{L^\infty(U)} \leq c_2 \sqrt{A} \, \|u\|_{W_2^{2,1}(U_A)} \,.
    \end{align}
\end{lemma}
\begin{proof}
    We apply a version of the Gagliardo–Nirenberg interpolation inequality to compare the $L^\infty$ norm to the $H^1$ norm on a bounded interval $U$. The specific form used can be found in the Comments on Chapter 8 of \cite{brezis2011functional}, setting $p=\infty, q=r=2, a=1/2$. There exists a constant $c_2$ such that for any $v \in H^1(U)$
    \begin{align*}
        \|v\|_{L^\infty(U)} \leq c_2 \|v\|_{L^2(U)}^{1/2} \|u\|_{H^1(U)}^{1/2} \,.
    \end{align*}
    Note that $\|v\|_{L^2(U)} \leq \|v\|_{H^1(U)}$ so we can replace the above with
    \begin{align}
        \|v\|_{L^\infty(U)} \leq c_2 \|v\|_{H^1(U)} \,. \label{Eqn: GN inequality on U}
    \end{align}
    Applying this to $\underline{u}(a)$ for almost every $0\leq a \leq A$ gives
    \begin{align} \label{Eqn: Applying GN inequality inside integral}
        \int_0^A \| \underline{u}(a) \|_{L^\infty(U)} \, da \leq c_2 \int_0^A \| \underline{u}(a) \|_{H^1(U)} \, da < \infty \,,
    \end{align}
    since $\underline{u} \in L^1(0,A;H^1(U))$. Hence $\underline{u}$ is integrable in $L^\infty(U)$ and so, combining the following inequalities gives
    \begin{align*}
        \bigg\| \int_0^A \underline{u}(a) \, da \bigg\|_{L^\infty(U)} 
        &\leq \int_0^A \| \underline{u}(a) \|_{L^\infty(U)} \, da &\text{\textsection 2 Thm. 4 of \cite{diestel1977vector}}\\
        &\leq c_2 \int_0^A \| \underline{u}(a) \|_{H^1(U)} \, da & \text{by \eqref{Eqn: Applying GN inequality inside integral}} \\
        &\leq c_2 \sqrt{A} \| \underline{u} \|_{L^2(0,A;H^1(U))} \, & \text{by \eqref{Eqn: bounding L1(0,T;H1) by L2(0,T;H1)}}\\
        &\leq c_2 \sqrt{A} \|u\|_{W_2^{2,1}(U_A)} \, & \text{by \eqref{Eqn: bound on u in $L^2(0,T;H1(U))$}}
    \end{align*}
    thus establishing the estimate. 
\end{proof}

We can now combine Lemma \ref{Lemma: Bounding lambda in L infinity with the norm of rho in Sob} with an estimate on the solution $\rho$ of \eqref{Eqn: lambda steady state} to bound $\mathcal{F}(\lambda)$ uniformly in $L^\infty(U)$. Specifically we show:
\begin{proposition} \label{Prop: F uniformly bounded}
    Equation \eqref{Eqn: Definition of F mapping} defines a mapping $\mathcal{F}:L^\infty(U) \rightarrow L^\infty(U)$. Moreover this mapping is uniformly bounded in $L^\infty(U)$ as there exists some constant $C$ such that for any $\lambda \in L^\infty(U)$,
    \begin{align} \label{Eqn: uniform L infty bound on F(lambda)}
        \| \mathcal{F}(\lambda) \|_{L^\infty(U)} \leq C \,.
    \end{align}
\end{proposition}
\begin{proof}
    Let $\rho \in W_2^{2,1}(U_A)$ be the unique solution of \eqref{Eqn: lambda steady state} for a given $\lambda \in L^\infty(U)$ and $\underline{\rho}$ the corresponding function in $L^1(0,A;H^1(U))$. Then  
    \begin{align*}
        \| \mathcal{F}(\lambda) \|_{L^\infty(U)} 
        &= \bigg\| \int_0^A \varM(a)\,\underline{\rho}(a) \, da \,\bigg\|_{L^\infty(U)} \\
        &\leq \varM_{\max} \bigg\| \int_0^A \underline{\rho}(a) \, da \,\bigg\|_{L^\infty(U)} \\
        &\leq c_2 \sqrt{A} \varM_{\max}\, \| \rho \|_{W_2^{2,1}(U_A)} \,, & \text{by Lemma \ref{Lemma: Bounding lambda in L infinity with the norm of rho in Sob}}\\
        &\leq c_1 c_2 \sqrt{A} \varM_{\max}\, \| \mu \|_{L^2(U)} \,, & \text{by \eqref{Eqn: bound on rho in Sob space}}.
    \end{align*}
    Letting $C = c_1 c_2 \sqrt{A} \varM_{\max}\, \| \mu \|_{L^2(U)}$ completes the proof.
\end{proof}

Next we use properties of the PDE \eqref{Eqn: lambda steady state} to look for a subset $K \subset L^\infty(U)$ that is invariant under $\mathcal{F}$, with the goal of applying Schauder's Fixed Point Theorem to conclude the existence of a fixed point of $\mathcal{F}$ in $K$. Applying a maximum principle, for example Lemma 5 in Chapter 2 of \cite{friedman2008partial}, gives that $\rho \geq 0$ and so $\lambda \geq 0$ (almost everywhere). As \eqref{Eqn: lambda steady state} preserves mass we also have that $\|\lambda\|_1 = \|\mu\|_1 = 1$. Hence, recalling \eqref{Eqn: uniform L infty bound on F(lambda)}, the set $K \subset K_1$ given by 
\begin{align} \label{Eqn: Definition of K}
    K = \{ \lambda \in L^\infty(U) : \lambda \geq 0 \text{ a.e.}, \| \lambda \|_{L^1(U)}=1, \, \| \lambda \|_{L^\infty(U)}\leq C \}
\end{align}
is invariant under $\mathcal{F}$. A major obstacle is that, in the norm topology on $L^\infty(U)$, it is not clear if the set $K$ is compact. Alternatively we may be able to apply Schauder's Fixed Point Theorem in the weak* topology on $L^\infty(U)$. The following Proposition \ref{Prop: K weakly compact and weakly closed} and Proposition \ref{Prop: F weakly continuous} guarantee the necessary properties. 

\begin{proposition} \label{Prop: K weakly compact and weakly closed}
    The set $K$ is weakly compact and weakly closed in $L^\infty(U)$. \textit{(Proof in Appendix \ref{Appendix: Proofs}).}
\end{proposition}
We next show that the mapping $\mathcal{F}$ is continuous in the weak* topology. To do so we first require the following Lemmas, the proofs of which can be found in Appendix \ref{Appendix: Proofs}. 

\begin{lemma} \label{Lemma: Lipschitz continuity of Lambda}
    There exist constants $\ell_1, \ell_2 \geq 0$ such that for any $\lambda \in K$ the corresponding $\Lambda$ and $\Lambda'$ are Lipschitz continuous with Lipschitz constants $\ell_1, \ell_2$ respectively. 
\end{lemma}

\begin{lemma} \label{Lemma: Pointwise plus Lipschitz implies L infty}
    Let $U\subset\mathds{R}$ be a bounded interval and $f^n$ a sequence of functions $f^n:\overline{U}\rightarrow\mathds{R}$ that are uniformly Lipschitz continuous with Lipschitz constant $L$. Assume that $f^n$ converges pointwise everywhere in $\overline{U}$ to a function $f:\overline{U}\rightarrow\mathds{R}$ that is also Lipschitz continuous with Lipschitz constant $L$. Then $f^n$ converges to $f$ in $L^\infty(U)$. 
\end{lemma}

\begin{proposition} \label{Prop: F weakly continuous}
    The mapping $\mathcal{F}:K \rightarrow K$ is continuous in the weak* topology on $L^\infty(U)$. 
\end{proposition}
\begin{proof} (Outline)
    The full proof can be found in Appendix \ref{Appendix: Proofs}.

    Let $\lambda^n$ be a sequence in $K$ that converges weakly to $\lambda \in L^\infty(U)$. Since $K$ is weakly closed, $\lambda \in K$. As for any $x\in \overline{U}$ the functions $\xi(y) = \phi(y-x)\,(y-x)$ and $\xi(y) = -\phi'(y-x)\,(y-x) - \phi(y-x)$ are both elements of $L^1(U)$, $\Lambda^n(x)$ and $(\Lambda^n)'(x)$ converge pointwise to $\Lambda(x)$ and $\Lambda'(x)$ respectively everywhere in $\overline{U}$. Applying Lemmas \ref{Lemma: Lipschitz continuity of Lambda} and \ref{Lemma: Pointwise plus Lipschitz implies L infty} gives that $\Lambda^n(x)$ and $(\Lambda^n)'(x)$ converge to $\Lambda(x)$ and $\Lambda'(x)$ respectively in $L^\infty(U)$. 

    For each $n \in \mathds{N}$ let $\rho^n \in W_2^{2,1}(U_A)$ denote the solution to \eqref{Eqn: lambda steady state} using interaction density $\lambda^n$ and similarly let $\rho\in W_2^{2,1}(U_A)$ be the solution using $\lambda$. By forming an initial-boundary value problem for $\rho^n - \rho$ and again applying Theorem 9.1 from Chapter IV Section 9 of \cite{ladyzhenskaia1968linear} we show that the convergence of $\Lambda^n$ and $(\Lambda^n)'$ implies that $\rho^n$ converges to $\rho$ in $W^{2,1}_2(U_A)$. 

    We may then apply Lemma \ref{Lemma: Bounding lambda in L infinity with the norm of rho in Sob} to show that $\mathcal{F}(\lambda^n)$ converges to $\mathcal{F}(\lambda)$ in the norm topology on $L^\infty(U)$ and therefore also in the weak topology on $L^\infty(U)$. 
\end{proof}

This concludes all the necessary steps to show the following Theorem. 
\begin{theorem} \label{Thm: Existence of steady states}
    Let Assumption \ref{Assumption group: Standard} and Assumption \ref{Assumption group: Fixed point} hold and assume that $\mu \in L^2(U)$. Then there exists at least one solution $\rho\in W_2^{2,1}(U_A)$ to the steady state equation \eqref{Eqn: Steady state}. 
\end{theorem}
\begin{proof}
    Each solution to the steady state equation can be uniquely characterised by a fixed point of the mapping $\mathcal{F}$ defined in \eqref{Eqn: Definition of F mapping}. We have shown the set $K\subset L^\infty(U)$, defined in \eqref{Eqn: Definition of K}, is invariant under $\mathcal{F}$. The set $K$ is also convex and, by Proposition \ref{Prop: K weakly compact and weakly closed}, weakly compact and weakly closed. In addition Proposition \ref{Prop: F weakly continuous} shows that $\mathcal{F}$ is weakly continuous on $K$. Hence applying Schauder's Fixed Point Theorem to $\mathcal{F}:K\rightarrow K$ with the weak topology on $L^\infty(U)$ gives the existence of at least one fixed point of $\mathcal{F}$ in the set $K$, and thus the existence of at least one solution to the steady state equation \eqref{Eqn: Steady state}. 
\end{proof}

A natural question to ask next is whether such a steady state is unique. More specifically, when $\tau$ is large individuals are quickly replaced and the opinion distribution cannot deviate significantly from the age-zero distribution $\mu$, so we may expect uniqueness of the steady state. Indeed the following Proposition \ref{Prop: F Lipschitz continuous}, proven in Appendix \ref{Appendix: Proofs}, gives that the mapping $\mathcal{F}$ is Lipschitz continuous on $K$. 

\begin{proposition} \label{Prop: F Lipschitz continuous}
    Under the assumptions of Theorem \ref{Thm: Existence of steady states} the mapping $\mathcal{F}:K\rightarrow K$ is Lipschitz continuous. Specifically there exists some constant $c$ depending on $U$, $\phi$, $A$ and $\tau$ but independent of $\lambda$, such that for all $\lambda_1,\lambda_2 \in K$, 
    \begin{align} \label{Eqn: Lipschitz continuity of F}
        \big\| \mathcal{F}(\lambda_1) - \mathcal{F}(\lambda_2) \big\|_{L^\infty(U)} 
        \leq  c \, \big(A^{1/2}+ A^{-1/4} \big)\, \varM_{\max} \, \big\|\lambda_1 - \lambda_2\big\|_{L^\infty(U)} \,.
    \end{align}
\end{proposition}
However, it is not possible to bound $c_8$ such that this Lipschitz constant is smaller than $1$ as $\tau\rightarrow\infty$. Therefore we cannot conclude that $\mathcal{F}$ is a contraction for sufficiently large $\tau$ (which would allow an application of the Banach Fixed-Point Theorem). Section \ref{Section: non-uniqueness} explores numerically an example in which steady states are and are not unique as $\tau$ varies. 

\subsection{Stationary age-zero density} \label{Section: Stationary age-zero distribution}

We first describe a specific situation in which it is straightforward to identify a steady state. Assume that the age zero distribution $\mu$ is a steady state of the classical mean-field PDE. That is,
\begin{align} \label{Eqn: steady state of classical MFL}
    \partial_x \Bigg( \mu(x) \bigg( \int_U \varphi(y - x) \, \mu(x) \, dy \, \bigg) \Bigg) - \, \frac{\sigma^2}{2} \partial_{x^2} \mu(x) &= 0 \,,
\end{align}
coupled with the appropriate no-flux boundary conditions. Define $\rho^*(a,x) = \mu(x)$ for all $a \in [0,A]$. Then we have,
\begin{align*}
    & \tau \partial_a\rho^*(a,x) + \partial_x \Bigg( \rho^*(a,x) \bigg( \int_U \int_0^A \varM(a,b) \, \varphi(y - x) \, \rho^*(b,y) \, db \, dy \, \bigg) \Bigg) - \, \frac{\sigma^2}{2} \partial_{x^2} \rho^*(a,x) \\
    &= \tau \partial_a\mu(x) + \partial_x \Bigg( \mu(x) \bigg( \int_U \int_0^A \varM(a,b) \, \varphi(y - x) \, \mu(y) \, db \, dy \, \bigg) \Bigg) - \, \frac{\sigma^2}{2} \partial_{x^2} \mu(x) \\
    &= \partial_x \Bigg( \mu(x) \bigg( \int_U \, \varphi(y - x) \, \mu(y) \, dy \bigg) \bigg( \int_0^A \varM(a,b) \, db \, \bigg) \Bigg) - \, \frac{\sigma^2}{2} \partial_{x^2} \mu(x) \\
    &= \partial_x \Bigg( \mu(x) \bigg( \int_U \, \varphi(y - x) \, \mu(y) \, dy \bigg) \Bigg) - \, \frac{\sigma^2}{2} \partial_{x^2} \mu(x) \\
    &= 0 \,,
\end{align*}
hence \eqref{Eqn: Steady state PDE} is satisfied. In addition
\begin{align*}
    \rho^*(0,x) - \mu(x) \, \int_U \rho^*(1,y) \, dy  &= \mu(x) - \mu(x) \, \int_U \rho^*(1,y) \, dy = \mu(x) - \mu(x) = 0\,
\end{align*}
hence \eqref{Eqn: Steady state BC} is satisfied. Finally, the no-flux boundary conditions for $\mu$ directly imply those for $\rho^*$, hence all conditions are satisfied for $\rho^*$ to be a steady state. We now use this context to provide a specific example of non-uniqueness and demonstrate the impact of changing $\tau$. 

\subsection{Non-uniqueness} \label{Section: non-uniqueness}

As discussed at the end of Section \ref{Section: Existence of steady states} we cannot show that the steady state is unique. That is, for given $\mu$ and $\phi$, there may be multiple $\lambda\in K$ such that the corresponding solution $\rho \in W_2^{2,1}(U_A)$ of \eqref{Eqn: lambda steady state} also satisfies $\lambda = \mathcal{F}(\lambda)$. In Figure \ref{fig:nonunique_stationary_distribution} below we give a specific example in the setting discussed in the previous Section \ref{Section: Stationary age-zero distribution}, in which $\mu$ is stationary in the classical mean-field limit. 

Consider an age-zero distribution $\mu^{(2)}$ satisfying \eqref{Eqn: steady state of classical MFL} in which the population is split into two equal clusters (as in e.g. \cite{goddard2022noisy}). As in the previous Section \ref{Section: Stationary age-zero distribution} define $\rho^*(a,x) = \mu^{(2)}(x)$ for all $a\in[0,A]$. As this is a steady state we also have that $\mathcal{F}(\mu^{(2)}) = \mu^{(2)}$. 

Denote by $\mu^{(1)}$ a distribution satisfying \eqref{Eqn: steady state of classical MFL} with only a single cluster. We continue to use the $\mu^{(2)}$ as the age-zero distribution but instead take $\lambda = \mu^{(1)}$. Using a numerical approximation of $\mathcal{F}$, described in Appendix \ref{Appendix: Numerics for steady states}, we begin from $\lambda_0 = \mu^{(1)}$ and iteratively apply this mapping. We find that we do not converge towards $\mu^{(2)}$ but instead reach a different steady state in which the population is encouraged towards consensus at higher ages.

An example is shown in Figure \ref{fig:nonunique_stationary_distribution}, using a smoothed bounded confidence interaction function with $r_1 = 0.5, r_2 = 0.6$, $\tau = 0.15$, $\sigma = 0.05$ and $M\equiv1$. Heatmaps show the stationary distribution over $U_A$, with age on the horizontal axis and opinion on the vertical axis. In the left panel $\lambda_0 = \mu^{(2)}$ and the population remains in these two stable clusters. However in the right panel $\lambda$ is given by the fixed point after iterating from $\lambda_0 = \mu^{(1)}$ and a different stationary distribution is found in which the population merges into a single cluster as individuals age. This shows that in certain cases multiple stable behaviours could be observed on the macroscopic level.  

\begin{figure}[ht!]
    \centering
    \includegraphics[width=1\linewidth]{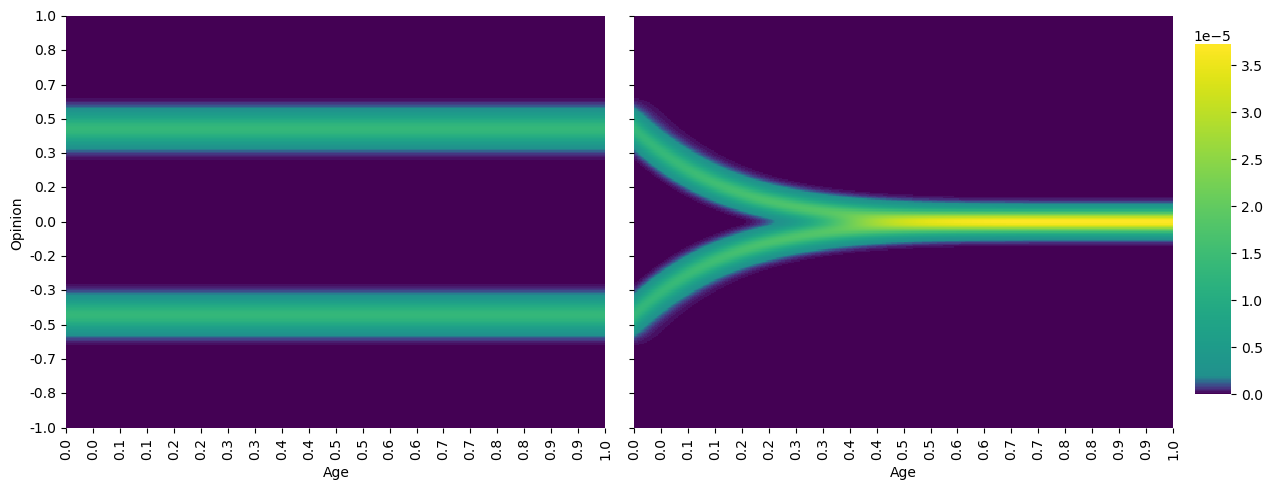}
    \caption{Showing two alternative steady states of \eqref{Eqn: PDE with ageing} over the joint age-opinion space. In both cases the interaction function $\phi$ is a smooth bounded confidence \eqref{eqn: phi for examples} with $r_1 = 0.5, r_2 = 0.6$, the diffusion coefficient $\sigma = 0.05$, the age-interaction kernel $M \equiv 1$, and the age-zero distribution $\mu = \mu^{(2)}$. The first corresponds to a fixed point of the mapping \eqref{Eqn: Definition of F mapping} with $\lambda_0 = \mu^{(2)}$, while the second corresponds to a fixed point in which $\lambda_0= \mu^{(1)}$ has a single cluster.}
    \label{fig:nonunique_stationary_distribution}
\end{figure}

We next investigate the effect of increasing $\tau$. For each value of $\tau$ we begin at $\lambda_0 = \mu^{(1)}$ and apply the numerical approximation of $\mathcal{F}$ until a fixed point is reached. The final value of $\lambda$ is shown for various $\tau$ in Figure \ref{fig: changing tau}. For small values below $\tau = 0.23$ (including the example shown in the right of Figure \ref{fig:nonunique_stationary_distribution}), $\lambda$ has a single peak, indicating that the population merges into a single cluster as individuals age. As $\tau$ is increased the age at which individuals merge increases, causing $\lambda$ to spread out. Eventually this merging occurs extremely close to age $a=1$ and $\lambda$ has two peaks. Finally for $\tau$ above approximately $0.323$ the stationary state in which the clusters merge no longer exists and $\lambda = \mu^{(2)}$ is the unique stationary state.

\begin{figure}[ht!]
    \centering
    \includegraphics[width=.8\linewidth, trim = {1cm 0cm 1.8cm 1cm}, clip]{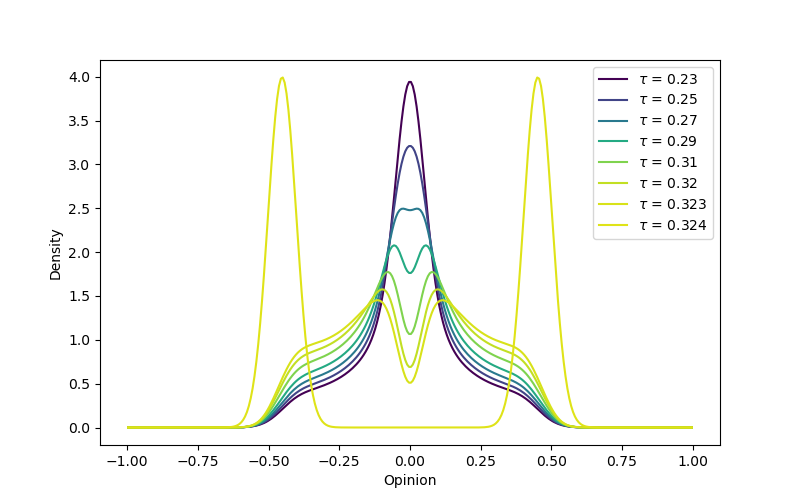}
    \caption{The fixed point $\lambda$ of the mapping $\mathcal{F}$ \eqref{Eqn: Definition of F mapping}, calculated numerically for various values of $\tau$. The interaction function $\phi$ is a smooth bounded confidence \eqref{eqn: phi for examples} with $r_1 = 0.5, r_2 = 0.6$, the diffusion coefficient $\sigma = 0.05$, the age-interaction kernel $M \equiv 1$, and the age-zero distribution $\mu = \mu^{(2)}$. As $\tau$ is increased there is a transition from a density with a single cluster to a density with two nearby peaks, then further to $\lambda = \mu^{(2)}$.}
    \label{fig: changing tau}
\end{figure}

This example raises the interesting question of what clustering behaviours, merges and emergences are possible as the opinion profile changes (in $a$) across a steady state? This forms a part of a broader question about the type of patterns it is possible to observe in the joint age-opinion distribution at a steady state, and which of these are stable. 

\subsection{Symmetry} \label{Section: Symmetry of steady states}

In this Section we show that if $\mu(x)$ is symmetric around zero then so is the steady state. 

Let $\rho(a,x)$ be a solution to \eqref{Eqn: Steady state}. Assume that for all $x\in[-1,1]$ that $\rho(0,x) = \mu(x) = \mu(-x)$. 

Define $\Tilde{\rho}(a,\Tilde{x}) = \rho(a,-x)$. Then
\begin{align*}
    \partial_a \Tilde{\rho}(a,x) &= \partial_a \rho(a,-x) \\
    &= - \partial_x \Bigg( \rho(a,-x) \bigg( \int_U \int_0^A \varM(a,b) \, \varphi(y + x)\, \rho(b,y) \, db \, dy \, \bigg) \Bigg) + \, \frac{\sigma^2}{2} \partial_{x^2} \rho(a,-x) \\
    &= - \partial_x \Bigg( \rho(a,-x) \bigg( \int_U \int_0^A \varM(a,b) \, \varphi(-y + x) \, \rho(b,-y) \, db \, dy \, \bigg) \Bigg) + \, \frac{\sigma^2}{2} \partial_{x^2} \rho(a,-x) \\
    &= - \partial_x \Bigg( -\rho(a,-x) \bigg( \int_U \int_0^A \varM(a,b) \, \varphi(y-x)\, \rho(b,-y) \, db \, dy \, \bigg) \Bigg) + \, \frac{\sigma^2}{2} \partial_{x^2} \rho(a,-x) \\
    &= - \partial_x \Bigg( -\Tilde{\rho}(a,\Tilde{x}) \bigg( \int_U \int_0^A \varM(a,b) \, \varphi(y-x) \, \Tilde{\rho}(b,y) \, db \, dy \, \bigg) \Bigg) + \, \frac{\sigma^2}{2} \partial_{x^2} \Tilde{\rho}(a,\Tilde{x}) \\
    &= - \partial_{\Tilde{x}} \Bigg( \Tilde{\rho}(a,\Tilde{x}) \bigg( \int_U \int_0^A \varM(a,b) \, \varphi(y-x) \, \Tilde{\rho}(b,y) \, db \, dy \, \bigg) \Bigg) + \, \frac{\sigma^2}{2} \partial_{x^2} \Tilde{\rho}(a,\Tilde{x}) \,.
\end{align*}
The condition on $\mu$ also ensures that the boundary conditions are the same. Hence the derivative (in $a$) of $\rho(a,x)-\rho(a,-x)$ is equal to zero everywhere in $U$ and so $ \rho(a,-x) = \Tilde{\rho}(a,\Tilde{x}) = \rho(a,x)$, meaning that \eqref{Eqn: Steady state} preserves the symmetry around $x=0$ present in $\mu(x)$. Hence if $\mu(x)$ is symmetric then the steady state is also symmetric.  

\section{Macroscopic Dynamics} \label{Section: PDE dynamics}

To shed light on the dynamics of the full system \eqref{Eqn: PDE with ageing} we show in this section key properties of the density $\rho$, then consider in Section \ref{Section: Connection to other models} several scenarios in which solutions can be expressed in terms of simpler systems. For the duration of Section \ref{Section: PDE dynamics} and Section \ref{Section: Connection to other models} we assume the solution $\rho(t,a,x)$ of \eqref{Eqn: PDE with ageing} is sufficiently smooth to allow the exchange of integrals and derivatives as required.

\subsection{Age transport and Conservation of Mass} \label{Section: Age transport}

Define the density 
\begin{align*}
    \pi(t,a) = \int_U \rho(t,a,x) \, dx
\end{align*}
which describes the total density of individuals with age $a$ (across all opinions), which formally satisfies
\begin{align*}
    \partial_t \pi(t,a) 
    &= \int_U \partial_t \rho(t,a,x) \, dx \\
    &= - \tau \partial_a \int_U \rho(t,a,x) \, dx - \int_U  \partial_x F[\rho](t,a,x)  \, dx \\
    &= - \tau \partial_a \pi(t,a) - F[\rho](t,a,-1) + F[\rho](t,a,1) \\
    &= - \tau \partial_a \pi(t,a)
\end{align*}
with the final equality holding due to the no-flux boundary conditions. Hence $\pi$ satisfies the transport equation
\begin{equation*}
    \partial_t \pi(t,a) + \tau \partial_a \pi(t,a) = 0 \,,
\end{equation*}
with periodic boundary condition $\pi(t,0) = \pi(t,1)$ arising from the age-zero boundary condition for $\rho$. Therefore
\begin{align*}
    \pi(t,a) = \pi(0, \{a - \tau t \} )
\end{align*}
where here $\{x\}$ denotes the fractional part of $x$. Thus if $\tau(0,a) = 1$ for all $a\in[0,1]$, meaning there is initially a uniform age density, then this will also be preserved. In addition, we can also see that 
\begin{align*}
    \int_U \int_0^A \rho(t,a,x) \, da \, dx 
\end{align*}
is constant in time and so the total mass is conserved. 

We note also that a similar argument to that in Section \ref{Section: Symmetry of steady states} can be applied to the full system to show that if $\rho(0,a,x)=\rho(0,a,-x)$ and $\mu(x) = \mu(-x)$ for all $a\in[0,1]$ and all $x\in[-1,1]$, then $\rho(t,a,x)=\rho(t,a,-x)$ for all $t \geq 0$. That is, the PDE \eqref{Eqn: PDE with ageing} preserves the symmetry around $x=0$.

\subsection{Evolution of the Mean Opinion}

Based on the observation of a shifting consensus in Figure \ref{fig:SDE moving consensus}, we next wish to consider how the mean opinion evolves in time. Define the mean opinion at age $a$ and the overall mean opinion respectively by
\begin{align}
    m(t,a) &= \int_U x \, \rho(t,a,x) \, dx \,, \\
    \Bar{m}(t) &= \int_0^A m(t,a) \, da = \int_0^A \int_U x \, \rho(t,a,x) \, dx \, da \,.
\end{align}
Then we have 
\begin{align*}
    \frac{d\Bar{m}}{dt} 
    &= \int_0^A \int_U x \, \partial_t\rho(t,a,x) \, dx \, da \\
    &= - \int_0^A \int_U x \, \Bigg[ \tau \partial_a\rho + \partial_x F[\rho](t,a,x) \Bigg] \, dx \, da \\
    &= - \tau \int_U x \, \bigg( \int_0^A  \partial_a\rho(t,a,x) \, da \bigg) \, dx  - \int_0^A \bigg( \int_U x \, \partial_x F[\rho](t,a,x) \, dx \bigg) \, da \\
    &= - \tau \int_U x \, \bigg( \rho(t,0,x) - \rho(t,1,x) \bigg) \, dx 
    - \int_0^A \bigg( x \, F[\rho](t,a,x) \Big|_{x=-1}^{x=1} - \int_U F[\rho](t,a,x) \, dx \bigg) \, da \\
    &= \tau \Big( m(t,0) - m(t,1) \Big)
    + \int_0^A \int_U F[\rho](t,a,x) \, dx \, da \,.
\end{align*}
Note that the final equality holds as a result of the no-flux boundary conditions. Looking at the interaction term in the flux we have
\begin{align}
    \int_0^A \int_U & \rho(t,a,x) \bigg( \int_U \int_0^A \varM(a,b) \, \varphi(y - x) \, \rho(t,b,y) \, db \, dy \bigg) \, dx \, da \nonumber\\
    &= \int_0^A \int_0^A \int_U \int_U \varphi(y - x) \, \varM(a,b) \, \rho(t,b,y) \, \rho(t,a,x) \, dx \, dy \, da \, db \label{Eqn: Zero effect on mean for model w ageing}
\end{align}
By comparing the effect of the two transformations $(x,a)\rightarrow(-x,a)$, $(y,b)\rightarrow(-y,b)$ and $(x,a)\rightarrow(-y,b)$, $(y,b)\rightarrow(-x,a)$ we see that the integral \eqref{Eqn: Zero effect on mean for model w ageing} is equal to zero when $\varM(a,b) = \varM(b,a)$ for all $a,b\in[0,A]$, meaning ages affect each other symmetrically. In this case only the noise term remains in the flux and we obtain
\begin{align} \label{Eqn: Evolution of mean}
    \frac{d\Bar{m}}{dt} 
    &= \tau \bigg( \int_U x\,\mu(x)\,dx - m(t,A) \bigg) + \frac{\sigma^2}{2} \Bigg( \int_0^A \rho(t,a,-1) \, da - \int_0^A \rho(t,a,1) \, da \Bigg) \,.
\end{align}
The first term describes evolution of the mean due to deaths/births, which brings the mean closer towards that of the age-zero distribution $\mu$. The second term describes evolution of the mean due to interactions with the no-flux boundary conditions, which arise as individuals with opinions on the boundary can only move their opinions in one direction. These contributions help to explain the movement of clusters towards a more symmetric distribution when the mean of $\mu$ is zero, as is observed in Figure \ref{fig:PDE examples}. 

\section{Connection to other models} \label{Section: Connection to other models}

We next consider several scenarios in which a solution to the full problem \eqref{Eqn: PDE with ageing} can be constructed from solutions of a different, typically simpler, system.  

\subsection{Convergence to Consensus} \label{Section: Convergence to Consensus}

We evaluate in a simple scenario how quickly the population converges to consensus. Assume $\varM\equiv1$ and $\phi\equiv1$, so that all individuals interact equally regardless of age or opinion. Assume also that the initial distribution $\rho_0$ and the age-zero distribution $\mu$ are both symmetric. In this case the PDE \eqref{Eqn: PDE with ageing} reduces to
\begin{align*}
    \partial_t \rho(t,a,x) + \tau \partial_a\rho(t,a,x) + \partial_x \Bigg( \rho(t,a,x) \bigg( \int_U \int_0^A (y - x) \, \rho(t,b,y) \, db \, dy \, \bigg) \Bigg) - \frac{\sigma^2}{2}\partial_{x^2}\rho(t,a,x) &= 0 \,.
\end{align*}
We first calculate
\begin{align*}
    \int_U \int_0^A (y - x) \, \rho(t,b,y) \, db \, dy 
    &= \int_0^A \int_U y \, \rho(t,b,y) \, dy \, db - x \,  \int_U \int_0^A \rho(t,b,y) \, db \, dy \\
    &= \Bar{m}(t) - x \,.
\end{align*}
As $\rho_0$ and $\mu$ are both symmetric the solution will remain symmetric and so $\Bar{m}(t) = 0$ for all $t\geq 0$. The PDE therefore reduces further to
\begin{align} \label{eqn: reduced PDE for convergence to consensus scenario}
    \partial_t \rho(t,a,x) + \tau \partial_a\rho(t,a,x) + \partial_x \big( -x\,\rho(t,a,x) \big) - \frac{\sigma^2}{2}\partial_{x^2}\rho(t,a,x) &= 0 \,.
\end{align}

Note that we have now effectively decoupled each age, as interactions between different ages arose only via the mean $\Bar{m}$ which we know to be zero due to symmetry. Therefore the only role played by age is to transport the density forwards in $a$. If we momentarily ignore this age transport term in \eqref{eqn: reduced PDE for convergence to consensus scenario} we obtain the Fokker-Planck equation of an Ornstein–Uhlenbeck process \citep{gardiner1985handbook} (with reflecting boundary conditions). We use this connection to construct a solution to \eqref{eqn: reduced PDE for convergence to consensus scenario}. 

For a given density $\varrho(x)$, let $\mathcal{S}_t [\varrho](x)$ denote the weak solution at time $t\geq0$ of the following initial-boundary value problem
\begin{subequations} \label{Eqn: OU problem}
\begin{align} 
    \partial_t p(t,x) + \partial_x \big( -x\,p(t,a,x) \big) - \frac{\sigma^2}{2}\partial_{x^2}p(t,a,x) &= 0 & \text{ in } U \,,\\[0.5em]
    x\,p(t,a,x) + \frac{\sigma^2}{2}\partial_{x}p(t,a,x) &= 0 & \text{ on } \partial U = \{-1,1\} \,,\\[0.5em]
    p(0,x) &= \varrho(x) & \text{ in } U \,.
\end{align}
\end{subequations}
Then the solution to \eqref{eqn: reduced PDE for convergence to consensus scenario} is given by
\begin{align} \label{Eqn: Solution of reduced PDE for convergence to consensus scenario}
    \rho(t,a,x) = 
    \begin{cases}
        \mathcal{S}_{a\tau^{-1}} [\mu](x) & \text{ for } 0 \leq a \leq \tau t \,,\\
        \mathcal{S}_t [\rho_0(a-\tau t,\cdot)](x) & \text{ for } \tau t \leq a \leq A \,.
    \end{cases}
\end{align}
For `large' values of $a$ in the second case, $\rho(t,a,x)$ is given by the solution to \eqref{Eqn: OU problem} at the same time $t$ but applied to $\rho_0(a-\tau t,\cdot)$ to account for the age transport. For `smaller' values of $a$, $a-\tau t$ would be negative, meaning $\rho(t,a,x)$ is in fact the evolution of the age-zero distribution $\mu$ under \eqref{Eqn: OU problem} until time $a\tau^{-1}$ (since it has been this length of time since this distribution was at $a=0$). Requiring $\rho_0(0,x) = \mu(x)$ ensures that there is no jump at $t = a\tau^{-1}$. Note that we can construct the solution in each section separately and concatenate them because the evolution at each age is indepedent. 

For $t\geq \tau^{-1} A$ all values of $a\in[0,A]$ fall into the first case. As $\mathcal{S}_a [\mu](x)$ is independent of $t$ the solution is therefore stationary for $t\geq \tau^{-1} A$, with the stationary distribution being entirely independent of the initial distribution $\rho_0$. 

Using this solution we can also examine the evolution of the opinion variance by studying 
\begin{align*}
    v(t) = \int_U x^2 \, \mathcal{S}_t [\varrho](x) \, dx \,.
\end{align*}
Differentiating in time, then integrating by parts twice, we have
\begin{align*}
    \frac{dv}{dt} 
    &= \int_U x^2 \, \partial_x \Big( x \, \mathcal{S}_t [\varrho](x) + \frac{\sigma^2}{2} \partial_{x}\mathcal{S}_t [\varrho](x) \Big)\, dx \,,\\
    &= -2 \int_U x \, \Big( x \, \mathcal{S}_t [\varrho](x) + \frac{\sigma^2}{2} \partial_{x}\mathcal{S}_t [\varrho](x) \Big)\, dx \,,\\
    &= -2 v(t) - \sigma^2 \bigg( \mathcal{S}_t [\varrho](1) - \mathcal{S}_t [\varrho](-1)  -  \int_U \mathcal{S}_t [\varrho](x) \, dx \bigg) \,.
\end{align*}
As \eqref{Eqn: OU problem} also preserves symmetry, if $\varrho$ is symmetric then $\mathcal{S}_t [\varrho](1) = \mathcal{S}_t [\varrho](-1)$. Moreover the mass is preserved so if $\varrho$ has mass 1 then so does $\mathcal{S}_t [\varrho]$. Note that both these properties are true of the choices for $\varrho$ in \eqref{Eqn: Solution of reduced PDE for convergence to consensus scenario}. In such a case we conclude that $v(t)$ satisfies
\begin{align*}
    \frac{dv}{dt} 
    &= \sigma^2 -2v \quad\Rightarrow\quad
    v(t) = \frac{\sigma^2}{2} + \bigg( v(0) - \frac{\sigma^2}{2} \bigg) e^{-2t} \,.
\end{align*}
Hence we conclude that the opinion density $\rho(t,a,x)$ has mean zero and variance
\begin{align} \label{eqn: convergence to consensus variance}
    \int_U x^2 \, \rho(t,a,x) \, dx = 
    \begin{cases}
        \dfrac{\sigma^2}{2} + \Bigg( \displaystyle\int_U x^2 \, \mu(x) \, dx - \dfrac{\sigma^2}{2} \Bigg) e^{-2a\tau^{-1}} & \text{ for } 0 \leq a \leq \tau t \,,\\[1.5em]
        \dfrac{\sigma^2}{2} + \Bigg( \displaystyle\int_U x^2 \, \rho_0(a - \tau t,x) \, dx - \dfrac{\sigma^2}{2} \Bigg) e^{-2t} & \text{ for } \tau t \leq a \leq A \,.
    \end{cases}
\end{align} 

\begin{figure}[ht!]
    \centering
    \includegraphics[width=0.9\linewidth]{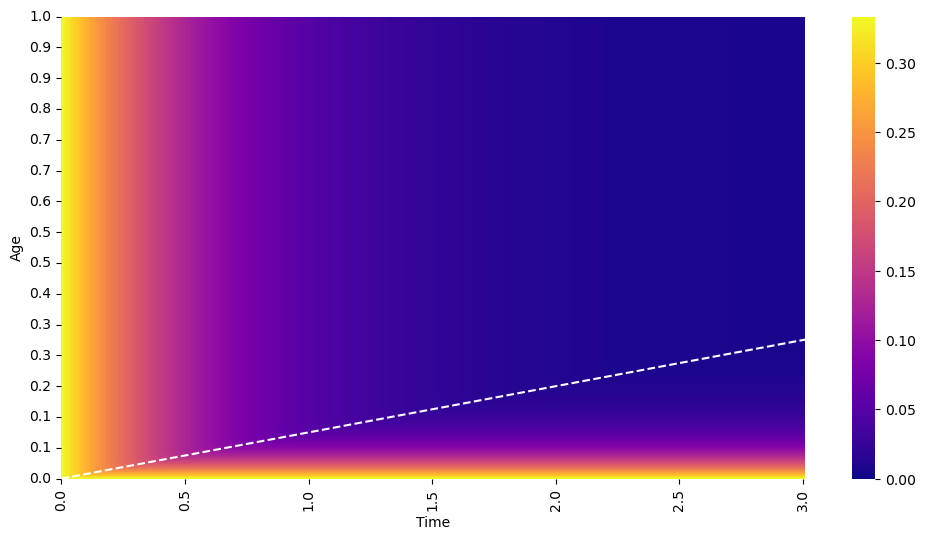}
    \caption{Opinion variance at each age over time, calculated from the numerical solution of \eqref{Eqn: PDE with ageing} with $\varM \equiv 1$, $\phi \equiv 1$, $\tau = \sigma = 0.1$ and both $\rho_0$ and $\mu$ uniformly distributed. A white dashed line shows when $a = \tau t$, which marks the transition between the two parts of the solution \eqref{Eqn: Solution of reduced PDE for convergence to consensus scenario}. }
    \label{fig: convergence to consensus variance}
\end{figure}

Figure \ref{fig: convergence to consensus variance} shows the opinion variance at each age over time, calculated from the numerical solution of \ref{Eqn: PDE with ageing} with $\varM \equiv 1$, $\phi \equiv 1$, $\tau = \sigma = 0.1$ and both $\rho_0$ and $\mu$ uniformly distributed. This matches exactly with the theoretical variance \eqref{eqn: convergence to consensus variance}. Initially there is a high variance at all ages due to the uniform initial condition but as the population moves towards consensus the only source of variance is from $\mu$, which decays as individuals age. 

\subsection{No inter-age interactions} \label{Section: No inter-age interactions}

We now consider the other extreme and assume that the age-interaction kernel $M(a,b) = \delta_a(b)$, that is individuals only interact with others of exactly the same age. Note that this violates Assumption \ref{Assumption group: Standard} and is not a reasonable choice for the microscopic model as all individuals almost surely have distinct ages, therefore there would simply be no opinion evolution. However, in the macroscopic setting, we will show that a solution for \eqref{Eqn: PDE with ageing} can still be constructed. 

For $M(a,b) = \delta_a(b)$ the flux becomes,
\begin{align*}
    F[\rho](t,a,x) 
    &= \rho(t,a,x) \bigg( \int_U \, \varphi(y - x) \, \rho(t,a,y) \, dy\bigg) - \, \frac{\sigma^2}{2} \partial_{x}\rho(t,a,x) \,,
\end{align*}
as each age $a$ only interacts with the density $\rho(t,a,x)$ at that same age. We therefore expect that the solution will match that of the standard mean-field model, transported forwards in age. We proceed in much the same way as in the previous Section \ref{Section: Convergence to Consensus}.

For a given density $\varrho(x)$, let $\mathcal{S}^{(2)}_t [\varrho](x)$ denote the weak solution at time $t\geq0$ of the classical mean-field limit, given by the following initial-boundary value problem
\begin{subequations}
\begin{align} 
    \partial_t p(t,x) + \partial_x G[p](t,x) &= 0 \\[0.5em]
    G[p](t,-1) = G[p](t,1) &= 0 \\[0.5em]
    p(0,x) &= \varrho(x)
\end{align}
\end{subequations}
with flux term
\begin{align}
    G[p](t,x) = p(t,x) \bigg( \int_U \, \varphi(y - x) \, p(t,y) \, dy\bigg) - \, \frac{\sigma^2}{2} \partial_{x}p(t,x) \,.
\end{align}
Then, similarly to the construction in Section \ref{Section: Convergence to Consensus}, the solution to \eqref{Eqn: PDE with ageing} for $M(a,b) = \delta_a(b)$ is given by
\begin{align}
    \rho(t,a,x) = 
    \begin{cases}
        \mathcal{S}^{(2)}_{a\tau^{-1}} [\mu](x) & \text{ for } 0 \leq a \leq \tau t \,,\\[0.5em]
        \mathcal{S}^{(2)}_t [\rho_0(a-\tau t,\cdot)](x) & \text{ for } \tau t \leq a \leq A \,.
    \end{cases}
\end{align}
Again the distribution is stationary for $t\geq \tau^{-1} A$ and independent of the initial distribution $\rho_0$. 

In both this example, where individuals interact only with those of exactly the same age, and the previous example of Section \ref{Section: Convergence to Consensus}, where all individuals interact at all times, we exploit the decoupling of ages to construct the solution of \eqref{Eqn: PDE with ageing} from the solution to simpler problems. 

\subsection{No ageing} \label{Section: No ageing}

In our next scenario we consider how the full problem \eqref{Eqn: PDE with ageing} can again be reduced in the case that $\tau = 0$. While this may appear to remove the ageing effect entirely it still allows for differing opinion distributions at different ages and certain non-trivial age interactions (specifically we may still reduce the problem in the case that age-interactions have the form $M(a,b) = M(b)$). 

Setting $\tau = 0$ in \eqref{Eqn: PDE with ageing} we instead have
\begin{subequations} \label{Eqn: tau=0 system}
\begin{align} 
    \partial_t \rho(t,a,x) +  \partial_x F[\rho](t,a,x) &= 0
\end{align}
\end{subequations}
with the same boundary conditions as in \eqref{Eqn: PDE with ageing}. 
Assume that $\varM(a,b) = \varM(b)$, in which case the flux \eqref{Eqn: Flux} reduces to
\begin{align*}
    F[\rho](t,a,x) = \rho(t,a,x) \Bigg( \int_U \varphi(y - x) \, \bigg( \int_0^A \varM(b) \, \rho(t,b,y) \, db \bigg) \, dy \Bigg) - \, \frac{\sigma^2}{2} \partial_{x}\rho(t,a,x)
\end{align*}
We first solve for
\begin{align*}
    u(t,x) = \int_0^A \varM(b) \, \rho(t,b,x) \, db \,,
\end{align*}
and use this to decouple the ages in \eqref{Eqn: tau=0 system}. Formally we have
\begin{align*}
    \partial_t u(t,x) 
    &= \int_0^A \varM(b) \,\partial_t \rho(t,b,x) \, db \\
    &= -\int_0^A \varM(b) \, \partial_x \Bigg( \rho(t,b,x) \bigg( \int_U \varphi(y - x) \, u(t,y) \, dy \bigg) \Bigg) - \, \varM(b) \frac{\sigma^2}{2} \partial_{x^2}\rho(t,b,x) \, db \,,\\
    &= - \partial_x \Bigg( \int_0^A \varM(b) \, \rho(t,b,x) \bigg( \int_U \varphi(y - x) \, u(t,y) \, dy \bigg) \, db  \Bigg) + \frac{\sigma^2}{2} \partial_{x^2} \bigg( \int_0^A \varM(b) \rho(t,b,x) \, db \bigg) \,.
\end{align*}
It is at this point crucial that $u(t,y)$ is independent of age, as we may now factorise the following integral,
\begin{align*}
    \int_0^A \varM(b) \, \rho(t,b,x) \bigg( \int_U \varphi(y - x) \, u(t,y) \, dy \bigg) \, db = \bigg( \int_0^A \varM(b) \, \rho(t,b,x) \, db \bigg) \bigg( \int_U \varphi(y - x) \, u(t,y) \, dy \bigg) \,.
\end{align*}
If $\varM(a,b)$ depended on both $a$ and $b$, rather than only on $b$, we would require $u$ to be a function of $(t,a,x)$ and the above factorisation would not be possible. Using this factorisation we conclude, 
\begin{align}
    \partial_t u(t,x) + \partial_x \Bigg( u(t,x) \bigg( \int_U \varphi(y - x) \, u(t,y) \, dy \bigg)  \Bigg) - \frac{\sigma^2}{2} \partial_{x^2} u(t,x) = 0 \,,
\end{align}
with appropriate no-flux boundary conditions. This is again the standard mean-field limit of the model without age structure. Using the notation of Section \ref{Section: No inter-age interactions} we have 
\begin{align*}
    u(t,x) = \mathcal{S}_t^{(2)}[u(0,\cdot)](x) = \mathcal{S}_t^{(2)}\bigg[\int_0^A M(b) \, \rho_0(a,\cdot)\bigg](x) \,.
\end{align*}
Since $u(t,x)$ is now known we may replace the flux term in \eqref{Eqn: tau=0 system} with
\begin{align*}
    F_u[\rho](t,a,x) = \rho(t,a,x) \bigg( \int_U \varphi(y - x) \, u(t,y) \, dy \bigg) - \, \frac{\sigma^2}{2} \partial_{x}\rho(t,a,x) \,.
\end{align*}
Using this flux term, the evolution of $\rho(t,a,x)$ is now independent of the value of $\rho$ at any other age. Hence we can instead specify a family of problems: for each $a \in (0,1]$ find a solution $\rho^{(a)}(t,x)$ of 
\begin{subequations} \label{Eqn: Family of problems for tau = 0}
\begin{align} 
    \partial_t \rho^{(a)}(t,x) +  \partial_x F_u\big[\rho^{(a)}\big](t,x)&= 0 \\[0.5em]
    F_u\big[\rho^{(a)}\big](t,-1) = F_u\big[\rho^{(a)}\big](t,1) &= 0 \\
    \rho^{(a)}(0,x) &= \rho_0(a,x)
\end{align}
\end{subequations}
while for $a = 0$ we have the constant
\begin{align*}
    \rho^{(0)}(t,x) &= \mu(x) \, \int_U \rho(0,1,y) \, dy \,.
\end{align*}
Similarly to the fixed point argument used to show the existence of steady states in Section \ref{Section: steady states} we have again reduced the problem to a (family of) linear second-order parabolic PDE(s). Since the problems \eqref{Eqn: Family of problems for tau = 0} differ only in their initial conditions $\rho_0(a,x)$ they will be continuous in $a$ if $\rho_0$ is, except at $a=0$ where the constant solution may create a discontinuity.  

Note that a similar analysis is not possible when $\tau > 0$ since the evolution of $u$ would then be influenced by the boundary at $a=A$. Specifically it would require knowledge of the opinion density at age $A$ as this would be replaced continuously by $\mu$ as individuals die, in the same way that the mean opinion at age $A$ is continuously replaced by the mean of $\mu$ in \eqref{Eqn: Evolution of mean}.

\subsection{McKendrick Ageing} \label{Section: MK ageing}

To conclude our analysis of the dynamics of \eqref{Eqn: PDE with ageing} we now compare against the ageing structure utilised in the McKendrick (MK) equation, a commonly used model for age-structured population dynamics in the life-sciences \citep{m1925applications,inaba2017age,keyfitz1997mckendrick}. The MK equation allows for more complex population dynamics than the simple age-transport in \eqref{Eqn: PDE with ageing}, therefore we are interested in cases where \eqref{Eqn: PDE with ageing} can still incorporate some of these effects. 

For a population of animals/individuals, the MK equation describes the density $\pi(t,a)$ of an age $a\in[0,\infty)$ at time $t\geq0$. The initial-boundary value problem reads
\begin{subequations} \label{MK-VF equation}
\begin{align}
    \partial_t \pi(t,a) + \tau \partial_a \pi(t,a) &= -\tau \, \vard(a) \, \pi(t,a) \,,\\
    \pi(t,0) &= \int_0^\infty \vard(a) \, \pi(t,a) \, da \,,\label{MK-VF equation: age zero BC}\\
    \pi(0,a) &= \pi_0(a) \,,
\end{align}
\end{subequations}
where $\vard:[0,\infty)\rightarrow\mathds{R}$ is the age-dependent death rate and $\pi_0\in L^\infty (\mathds{R}^+)$ is the initial age profile. Note that in the standard MK model the constant $\tau$ could be removed using a time rescaling, it is retained here to simplify the comparison with the opinion formation model \eqref{Eqn: PDE with ageing}. The boundary condition \eqref{MK-VF equation: age zero BC} is chosen to ensure a constant total population size. In addition we assume that 
\begin{align*}
    \int_0^\infty \pi_0(a) \, da = 1 \,,
\end{align*}
so that 
\begin{align*}
    \int_0^\infty \pi(t,a) \, da = 1 \,,
\end{align*}
for all $t\geq 0$. 

In the same way that \eqref{Eqn: PDE with ageing} can be viewed as an extension of the age transport equation to include opinion dynamics, this setup can now be similarly extended to give
\begin{subequations} \label{Eqn: PDE with MK ageing v1}
\begin{align} 
    \partial_t \rho(t,a,x) + \tau \partial_a \rho(t,a,x) + \, \partial_x \Tilde{F}[\rho](t,a,x) &= -\tau \, \vard(a) \, \rho(t,a,x)  \,,\\[0.5em]
    \Tilde{F}[\rho](t,a,-1) = \Tilde{F}[\rho](t,a,1) &= 0 \,,\\
    \rho(t,0,x) &= \mu(x) \int_U \int_0^\infty \vard(a) \, \rho(t,a,x) \, da \, dx \label{Eqn: PDE with MK ageing age-zero BC v1} \,,\\
    \rho(0,a,x) &= \rho_0(a,x) \,,
\end{align}
\end{subequations}
where the definition of $\Tilde{F}[\rho]$ is adapted slightly from \eqref{Eqn: Flux} as 
\begin{align} \label{Eqn: fluxes for MK ageing}
        \Tilde{F}[\rho](t,a,x) = \rho(t,a,x) \bigg( \int_U \int_0^\infty \varM(a,b) \, \varphi(y - x) \, \rho(t,b,y) \, db \, dy - \, \frac{\sigma^2}{2} \partial_{x}\log\big(\rho(t,a,x)\big) \bigg) \,.
\end{align}
In addition, for a solution $\rho$ of \eqref{Eqn: PDE with MK ageing v1}, the age density
\begin{align*}
    \pi(t,a) = \int_U \rho(t,a,x) \, dx
\end{align*}
satisfies the MK equation \eqref{MK-VF equation} with initial condition
\begin{align*}
    \pi_0(a) = \int_U \rho_0(a,x) \, dx \,,
\end{align*}
and $\pi$ can thus be solved independently of $\rho$ (e.g. using the method of characteristics). Using this solution we can replace update the boundary condition \eqref{Eqn: PDE with MK ageing age-zero BC v1} to obtain
\begin{subequations} \label{Eqn: PDE with MK ageing}
\begin{align} 
    \partial_t \rho(t,a,x) + \tau \partial_a \rho(t,a,x) +  \, \partial_x \Tilde{F}[\rho](t,a,x) &= -\tau \, \vard(a) \, \rho(t,a,x)  \,,\\[0.5em]
    \Tilde{F}[\rho](t,a,-1) = \Tilde{F}[\rho](t,a,1) &= 0 \,,\\
    \rho(t,0,x) &= \mu(x) \bigg( \int_0^\infty \vard(a) \, \pi(t,a) \, da \bigg) \label{Eqn: PDE with MK ageing age-zero BC} \,,\\
    \rho(0,a,x) &= \rho_0(a,x) \,.
\end{align}
\end{subequations}

This approach allows for greater flexibility in specifying the population's age profile, with the drawback that ages can become arbitrarily large. Indeed steady states of the MK equation \eqref{MK-VF equation} have the form 
\begin{align*}
    \pi(a) = \pi(0) \, \exp\bigg( -\int_0^a \vard(b) \, db \bigg) \,.
\end{align*}
So $\pi(a)$ will be compactly supported (on $[0,A]$) if and only if there exists some $A > 0$ such that 
\begin{align*}
    \lim_{a \rightarrow A^-} d(a) = \infty \,,
\end{align*}
that is the death rate diverges. The following Proposition \ref{Prop: Correspondence with MK}, proven in Appendix \ref{Appendix: Proofs}, shows that in such cases a solution to \eqref{Eqn: PDE with MK ageing} can be constructed from a solution to \eqref{Eqn: PDE with ageing}. That is, the original model \eqref{Eqn: PDE with ageing} with a uniform age profile can also capture the behaviour of the model \eqref{Eqn: PDE with MK ageing} which uses the more realistic MK ageing. 

\begin{proposition} \label{Prop: Correspondence with MK}
    Assume that $\mu, \vard, \rho_0$ and $\varM$ are given and let $\pi(a)$ be the corresponding stationary age profile of \eqref{MK-VF equation}, which is assumed to have compact support in $[0,1]$. Let $q$ be a solution of \eqref{Eqn: PDE with ageing} with age interactions given by $\varM'(a,b) = \varM(a,b) \, \pi(b)$. Then $\rho$ defined by 
    \begin{align} \label{Eqn: construction of MK solution}
        \rho(t,a,x) = q(t,a,x) \, \pi(a) 
    \end{align}
    is a solution of \eqref{Eqn: PDE with MK ageing}. 
\end{proposition}

Therefore, by incorporating the stationary age profile $\pi$ into the age interaction term $\varM'$ the original model \eqref{Eqn: PDE with ageing} can capture some effects from the model using McKendrick ageing \eqref{Eqn: PDE with MK ageing}, in particular a non-uniform age profile. It should be noted that this only applies in the case that $\pi$ is compactly supported, although in the context of opinion dynamics this is a realistic assumption. 

Note that if $\varM(a,b)$ is independent of $a$ then so is $\varM'(a,b)$, hence Assumption \ref{Assumption group: Fixed point} still holds and the analysis of steady states in Section \ref{Section: steady states} is still applicable. 

This raises the question of how a non-stationary age density $\pi(t,a)$ might affect the opinion formation process. For example, we observe in many numerical solutions to \eqref{Eqn: PDE with ageing} that the population becomes increasingly clustered as higher ages. Therefore a population with an increasing proportion of older individuals might become more influenced by these clusters, while a population with an increasing proportion of younger individuals might be increasingly influenced by the age-zero distribution $\mu$. 

\section{Conclusion}  \label{Section: Conclusion} 

Throughout this paper we have observed in both the microscopic SDE model as well as the macroscopic PDE model for opinion dynamics that introducing an explicit continuous age structure allows for a richer variety of opinion dynamics. The non-linearilty introduced by the birth and death process opens the possibility of a shifting consensus, the emergence of new clusters and periodic behaviour. Such new possibilities naturally generate several open questions and challenges. 

To begin, while the similarity in structure with the Cucker-Smale model and derivation performed in Section \ref{Section: Formal Derivation} indicate the convergence of the SDE model \eqref{Eqn: SDE model} to the PDE model \eqref{Eqn: PDE with ageing}, it remains to show this rigorously for this specific model. This would also establish the existence (and uniqueness) of the solution to \eqref{Eqn: PDE with ageing} in an appropriate space, which would allow the formal computations made in Section \ref{Section: PDE dynamics} and Section \ref{Section: Connection to other models} to be made rigorous. 

In demonstrating the existence of steady states and discussing their uniqueness we have begun to classify the model's behaviours, however there is still significant work to be done in this regard. For instance, to determine precise conditions for the uniqueness of steady states of \eqref{Eqn: PDE with ageing}. Moreover it would be useful to define an equivalence between steady states that describes the various patterns in Figure \ref{fig:PDE examples}. Finally it remains to prove that the repeated patterns observed in Figure \ref{fig:PDE wiggles} and Figure \ref{fig:PDE periodic consensus} are indeed periodic solutions, with the broader goal of determining when and what periodic solutions are possible. 

In considering the applicability of this model and its ability to predict real population-level age-structured opinion formation, there are various sources of data that may be useful. For example, the POLYMOD data set has been widely used as a measure of contact frequency between individuals of different ages \citep{mossong2008social,Mossong2017data}. As described in Section \ref{Section: MK ageing}, this age-interaction kernel could also be used to incorporate a fitted age distribution. In addition, even in the case that age-structured data is not available, model outputs could be aggregated over ages as in Figure \ref{fig:PDE examples} to provide a population-level opinion distribution that could be compared against survey data.  

Lastly, one could consider the introduction of noise on the PDE level, giving a stochastic partial differential equation (SPDE). Such an approach was considered in \cite{wehlitz2024approximating} to mimic the movement of whole clusters, as observed in Figure \ref{fig:merging and emerging clusters}, that is normally lost on the PDE level (as the stochasticity in the SDE becomes deterministic diffusion in the PDE). By using an SPDE one could recreate the movement of clusters and study their merging and re-emergence. 

While there are several open directions for future research, it is already clear that the introduction of a continuous age structure leads to novel opinion dynamics that are interesting from both an applied and mathematical perspective. We hope that this framework will contribute to the development of more realistic opinion formation models that capture, and begin to explain, the persistent heterogeneity of opinions seen in the real world. 

\section{Acknowledgements}

AN was supported by the Engineering and Physical Sciences Research Council through the Mathematics of Systems II Centre for Doctoral Training at the University of Warwick (reference EP/S022244/1). MTW acknowledges partial support from the EPSRC Small Grant EPSRC EP/X010503/1.

For the purpose of open access, the authors have applied a Creative Commons Attribution (CC-BY) license to Any Author Accepted Manuscript version arising from this submission.

The authors declare no competing interests. 

\bibliography{bibliography}
\bibliographystyle{agsm}

\clearpage
\appendix

\section{Proofs} \label{Appendix: Proofs}

\begin{lemma} \label{Lemma: underline well defined}
    For any $u \in W_2^{2,1}(U_A)$ the function $\underline{u}:[0,A]\rightarrow H^1(U)$ defined by $\underline{u}(a) = u(a,\cdot)$ is an element of $L^2(0,A;H^1(U))$. In addition 
    \begin{align} 
        \|\underline{u}\|_{L^2(0,A;H^1(U))} \leq \|u\|_{W_2^{2,1}(U_A)} \,.
    \end{align}
\end{lemma}
\begin{proof}
    As $u \in W_2^{2,1}(U_A)$, the required weak derivative $D_x \underline{u}(a)$ exists for almost all $a$. We then compute directly
    \begin{align*}
        \| \underline{u} \|_{L^2(0,A;H^1(U))} 
        &= \bigg( \int_0^A \|\underline{u}(a)\|_{H^1(U)}^2 \, da \bigg)^{1/2} \\
        &= \bigg( \int_0^A \|\underline{u}(a)\|_{L^2(U)}^2 
        + \|D_x \underline{u}(a)\|_{L^2(U)}^2 \, da \bigg)^{1/2} \\
        &= \bigg( \int_0^A \int_U u(a,x)^2 + \big(D_x u(a,x)\big)^2 \, dx \, da \bigg)^{1/2} \\
        &\leq \bigg( \int_0^A \int_U u(a,x)^2 \, dx \, da \bigg)^{1/2} + \bigg( \int_0^A \int_U \big(D_x u(a,x)\big)^2 \, dx \, da \bigg)^{1/2} \\
        &\leq \|u\|_{W_2^{2,1}(U_A)} 
    \end{align*}
    Hence $\underline{u} \in L^2(0,A;H^1(U))$ and we have shown \eqref{Eqn: bound on u in $L^2(0,T;H1(U))$}.
\end{proof}

\textbf{Proof of Proposition \ref{Prop: K weakly compact and weakly closed}}

\begin{proof}
    The set $K$ is a subset of 
    \begin{align*}
        K' = \{ \lambda \in L^\infty(U) : \| \lambda \|_{L^\infty(U)}\leq C \} \,,
    \end{align*}
    which itself is weakly compact. This is a consequence of the Banach-Alaoglu Theorem, see for example Theorem 3.14 and Corollary 3.15 in \cite{clarke2013functional}.

    Denote by $\langle \xi, u \rangle$ the pairing of $\xi \in L^1(U)$ with $u \in L^\infty(U)$. Let $\lambda^n$ be a sequence of elements of $K$ converging weakly in $L^\infty(U)$ to some $\lambda \in L^\infty(U)$. We wish to show that $\lambda \in K$. 
     
    Assume that $\lambda$ is not positive almost everywhere, then there exists some set $V \subset U$ with strictly positive Lebesgue measure on which $\lambda$ is strictly negative. Let $\chi_V$ denote the indicator function on this set, then $\chi_V \in L^1(U)$ and so $\langle \chi_V, \lambda^n \rangle \rightarrow \langle \chi_V, \lambda \rangle$. However for every $n \in \mathds{N}$ the integral $\langle \chi_V, \lambda^n \rangle \geq 0$,     since $\chi_V$ is positive and $\lambda^n$ is positive almost everywhere, while the limit $\langle \chi_V, \lambda \rangle < 0$ by construction. This is a contradiction, hence $\lambda \geq 0$ a.e. 
    
    As $\lambda \geq 0$ a.e. we have that 
    \begin{align*}
        \| \lambda \|_{L^1(U)} = \int_U |\lambda(x)| \, dx =  \int_U \lambda(x) \, dx = \langle \mathbf{1}, \lambda \rangle
    \end{align*}
    where $\mathbf{1}\in L^1(U)$ denotes the constant function with value $1$. Since $\lambda^n$ converges weakly to $\lambda$, 
    \begin{align*}
        \langle \mathbf{1}, \lambda \rangle 
        = \lim_{n\rightarrow\infty} \langle \mathbf{1}, \lambda^n \rangle
        = \lim_{n\rightarrow\infty} \| \lambda^n \|_{L^1(U)}
        = 1 \,.
    \end{align*}
    The final property required of $\lambda$ is $\| \lambda \|_{L^\infty(U)}\leq C$. This follows from the fact that the unit ball in $L^\infty(U)$, and furthermore any dilation of the unit ball, is weakly closed. Thus $\lambda$ satisfies all requirements to be an element of $K$ and so $K$ is weakly closed. 
\end{proof}

\textbf{Proof of Lemma \ref{Lemma: Lipschitz continuity of Lambda}}
\begin{proof}
    Recall that $\varphi(y-x):=\phi(y-x)\,(y-x)$. Since $\phi\in C^2(U)$ both $\varphi$ and $\varphi'$ are Lipschitz continuous with constants denoted $\ell_1$ and $\ell_2$ respectively. 
    
    Now for any $\lambda \in K$ and any $x,z\in \overline{U}$ consider,
    \begin{align*}
        |\Lambda(x) - \Lambda(z)| 
        &= \bigg| \int_U \varphi(y-x)\,\lambda(y)\,dy - \int_U \varphi(y-z)\,\lambda(y)\,dy \bigg| \,,\\
        &\leq \int_U |\varphi(y-x) - \varphi(y-z)|\,\lambda(y)\,dy \,,\\
        &\leq \int_U \ell_1 |x-z|\,\lambda(y)\,dy \,,\\
        &= \ell_1 |x-z|\,\int_U \lambda(y)\,dy \,,\\
        &= \ell_1 |x-z| \,.
    \end{align*}
    An identical arguments holds for $\Lambda'$ with Lipschitz constant $\ell_2$. 
\end{proof}

\textbf{Proof of Lemma \ref{Lemma: Pointwise plus Lipschitz implies L infty}}
\begin{proof}
    We show that $f^n$ converges to $f$ uniformly, from which convergence in $L^\infty(U)$ follows. Fix any $\varepsilon > 0$. For each $y \in \overline{U}$ let $B(y,\frac{\varepsilon}{4L})$ be the open ball of radius $\frac{\varepsilon}{4L}$ centred at $y$. As $\overline{U}$ is compact there exists a finite collection of points $y_1,\dots,y_m$ such that the union $\cup_{i=1}^m B(y_i,\frac{\varepsilon}{4L})$ covers $\overline{U}$. 

    For each $i=1,\dots,m$ take $x\in B(y_i,\frac{\varepsilon}{4L})$ and consider
    \begin{align*}
        |f^n(x) - f(x)| 
        &\leq |f^n(x) - f^n(y_i)| + |f^n(y_i) - f(y_i)| + |f(y_i) - f(x)| \\
        &\leq L|x - y_i| + |f^n(y_i) - f(y_i)| + L|y_i - x| \\
        &\leq \frac{\varepsilon}{2} + |f^n(y_i) - f(y_i)| 
    \end{align*}
    Since $f^n(y_i)$ converges to $f(y_i)$ there exists $N_i \in \mathds{N}$ such that for all $n\geq N_i$, $|f^n(y_i) - f(y_i)| \leq \frac{\varepsilon}{2}$. So for $n \geq N_i$ we have $|f^n(x) - f(x)| \leq \varepsilon$ for all $x\in B(y_i,\frac{\varepsilon}{4L})$.

    Define $N = \max_{i=1,\dots,N}$. Then for $n\geq N$ we have $|f^n(x) - f(x)| \leq \varepsilon$ for all $x\in \cup_{i=1}^m B(y_i,\frac{\varepsilon}{4L}) \supset \overline{U}$. Hence $f^n$ converges to $f$ uniformly in $\overline{U}$ and therefore also in $L^\infty(U)$. 
\end{proof}

\textbf{Proof of Proposition \ref{Prop: F weakly continuous}}
\begin{proof}
    Fix some $\lambda \in K$ and consider a sequence $\lambda^n \in K$ converging weakly to $\lambda$ in $L^\infty(U)$, meaning that for any $\xi \in L^1(U)$, 
    \begin{align*}
        \int_U \lambda^n(y) \, \xi(y) \, dy \rightarrow \int_U \lambda(y) \, \xi(y) \, dy \,.
    \end{align*}
    Therefore, by fixing some $x\in \overline{U}$ and taking $\xi(y) = \phi(y-x)\,(y-x)$ and $\xi(y) = -\phi'(y-x)\,(y-x) - \phi(y-x)$, we have that $\Lambda^n(x)$ and $(\Lambda^n)'(x)$ converge pointwise to $\Lambda(x)$ and $\Lambda'(x)$ respectively everywhere in $\overline{U}$. Applying Lemmas \ref{Lemma: Lipschitz continuity of Lambda} and \ref{Lemma: Pointwise plus Lipschitz implies L infty} gives that $\Lambda^n(x)$ and $(\Lambda^n)'(x)$ converge to $\Lambda(x)$ and $\Lambda'(x)$ respectively in $L^\infty(U)$. We now show that the corresponding solutions to \eqref{Eqn: lambda steady state} converge. 

    For each $n \in \mathds{N}$ let $\rho^n \in W_2^{2,1}(U_A)$ denote the solution to \eqref{Eqn: lambda steady state} using interaction density $\lambda^n$ and similarly let $\rho\in W_2^{2,1}(U_A)$ be the solution using $\lambda$. We will next show that $\rho^n$ converges to $\rho$ in $W_2^{2,1}(U_A)$. Define $q^n = \rho^n - \rho \in W_2^{2,1}(U_A)$. Combining the corresponding versions of \eqref{Eqn: lambda steady state} gives that $q^n$ is a weak solution to the following initial-boundary value problem:
    \begin{subequations} \label{Eqn: PDE for q^n}
    \begin{align}
        \tau \partial_a q^n(a,x) + \partial_x \Big( q^n(a,x) \Lambda^n(x) \Big) - \, \frac{\sigma^2}{2} \partial_{x^2} q^n(a,x) &= -\partial_x \bigg( \rho(a,x) \Big( \Lambda^n(x) - \Lambda(x) \Big) \bigg) \,, \\
        q^n(0,x) &= 0 \,,\\
        q^n(a,1) \, \Lambda^n(1) - \, \frac{\sigma^2}{2} \partial_{x} q^n(a,1) &= -\rho(a,1) \Big( \Lambda^n(1) - \Lambda(1) \Big) \,,\\
        q^n(a,-1) \, \Lambda^n(-1) - \, \frac{\sigma^2}{2} \partial_{x} q^n(a,-1) &= -\rho(a,-1) \Big( \Lambda^n(-1) - \Lambda(-1) \Big) \,.
    \end{align}    
    \end{subequations}
    We wish to apply Theorem 9.1 from Chapter IV Section 9 of \cite{ladyzhenskaia1968linear} to obtain a bound on $\| q^n \|_{W_2^{2,1}(U_A)}$ of the form 
    \begin{align*}
        \| q^n \|_{W_2^{2,1}(U_A)} \leq c_3 \Big( \| f^n \|_2 + \| \Phi^n \|_{W_2^{3/2,3/4}(S_A)} + A^{-3/4} \| \Phi^n \|_{L^2(S_A)} \Big) \,,
    \end{align*}
    where
    \begin{align*}
        f^n(a,x) &= -\partial_x \bigg( \frac{1}{\tau}\rho(a,x) \Big( \Lambda^n(x) - \Lambda(x) \Big) \bigg) \,, \\
        \Phi^n(a,x) &= -\rho(a,x) \Big( \Lambda^n(x) - \Lambda(x) \Big) \,.
    \end{align*}
    First consider $f^n$. Using H\"older's inequality we have the following bound,
    \begin{align}
        \| f^n \|_2 &= \frac{1}{\tau}\,\bigg\| \, \partial_x \rho(a,x) \Big( \Lambda^n(x) - \Lambda(x) \Big) +  \rho(a,x) \Big( (\Lambda^n)'(x) - \Lambda'(x) \Big) \bigg\|_2 \nonumber\\
        &\leq \frac{1}{\tau}\,\big\| \, \partial_x \rho \big\|_2 \, \big\| \Lambda^n - \Lambda \big\|_\infty +\frac{1}{\tau}\, \big\| \rho \big\|_2 \, \big\| (\Lambda^n)' - \Lambda' \big\|_\infty \, \\
        &\leq \frac{1}{\tau}\,\big\| \rho \big\|_{W_2^{2,1}(U_A)} \, \Big( \big\| \Lambda^n - \Lambda \big\|_\infty + \big\| (\Lambda^n)' - \Lambda' \big\|_\infty \Big) \,.\label{Eqn: bound on f^n for q^n}
    \end{align}
    We have already established that $\Lambda^n$ and $(\Lambda^n)'$ converge to $\Lambda$ and $\Lambda'$ respectively in $L^\infty(U)$ and therefore in $L^\infty(U_A)$ (as they are constant in $a$). Moreover, $\big\| \rho \big\|_{W_2^{2,1}(U_A)}$ is a finite constant since the solution $\rho \in W_2^{2,1}(U_A)$. Hence $\| f^n \|_2$ is finite and converges to zero as $n\rightarrow\infty$.
    
    Next we require that $\Phi^n \in W_2^{3/2,3/4}(S_A)$. Applying Lemma 3.4 from Chapter 2 Section 3 of \cite{ladyzhenskaia1968linear} gives the existence of a constant $c_4$ such that for any $u \in W^{2,1}_2(U_A)$ 
    \begin{align*}
        \| u \|_{W^{3/2,3/4}_2(S_A)} \leq c_4 \| u \|_{W^{2,1}_2(U_A)} \,.
    \end{align*}
    Since $\Lambda^n$ and $\Lambda$ are continuous and $\rho \in W_2^{2,1}(U_A)$, $\Phi^n \in W_2^{2,1}(U_A)$ for all $n \in \mathds{N}$. Hence $\Phi^n \in W_2^{3/2,3/4}(S_A)$ and, 
    \begin{align*}
        \| \Phi^n\|_{W^{3/2,3/4}_2(S_A)} 
        &\leq c_4 \| \Phi^n \|_{W^{2,1}_2(U_A)} \,,\\
        &\leq c_4 \bigg( \| \Phi^n \|_2 + \| D_a\Phi^n \|_2 + \| D_x \Phi^n \|_2 + \| D_{x^2} \Phi^n \|_2 \bigg) \,, \\
        &\leq c_5 \| \rho \|_{W^{2,1}_2(U_A)} \bigg( 
        \| \Lambda^n - \Lambda \|_\infty + 
        \| (\Lambda^n)' - \Lambda' \|_\infty + 
        \| (\Lambda^n)'' - \Lambda'' \|_\infty \bigg) \,, 
    \end{align*}
    with the final inequality obtained by applying H\"older's inequality to each component of $\| \Phi^n \|_{W^{2,1}_2(U_A)}$ and using $\| \rho \|_{W^{2,1}_2(U_A)}$ as an upper bound for the $L^2$ norm of $\rho$ and its weak derivatives. Using an identical argument as that for $\Lambda^n$ and $(\Lambda^n)'$, $(\Lambda^n)''$ converges to $\Lambda''$ in $L^\infty(U)$. Hence $\| \Phi^n\|_{W^{3/2,3/4}_2(S_A)} \rightarrow 0$. This argument also shows that $\| \Phi^n \|_{L^2(S_A)} \rightarrow 0 $. 
    
    Hence we can apply Theorem 9.1 from Chapter IV Section 9 of \cite{ladyzhenskaia1968linear} to give the existence of a unique solution $q^n \in W_2^{2,1}$ to \eqref{Eqn: PDE for q^n}. Moreover there exists a constant $c_3$ such that,
    \begin{align} \label{Eqn: bound on q^n}
        \| q^n \|_{W_2^{2,1}(U_A)} \leq c_3 \Big( \| f^n \|_2 + \| \Phi^n \|_{W_2^{3/2,3/4}(S_A)} + A^{-3/4} \| \Phi^n \|_{L^2(S_A)} \Big) \,,
    \end{align}
    so $\| q^n \|_{W_2^{2,1}(U_A)}\rightarrow0$ as $n\rightarrow\infty$, meaning that $\rho^n$ converges to $\rho$ in $W_2^{2,1}$. Applying Lemma \ref{Lemma: Bounding lambda in L infinity with the norm of rho in Sob} to $\rho^n - \rho$ we then have that 
    \begin{align*}
        \big\| \mathcal{F}(\lambda^n) - \mathcal{F}(\lambda) \big\|_{L^\infty(U)} 
        &\leq c_2 \sqrt{A} \varM_{\max}\, \| \rho^n - \rho \|_{W_2^{2,1}(U_A)} \rightarrow 0 \,.
    \end{align*}
    Hence $\mathcal{F}(\lambda^n)$ converges (in norm) to $\mathcal{F}(\lambda)$ in $L^\infty(U)$, and so also converges weakly in $L^\infty(U)$. Thus the mapping $\mathcal{F}:K\rightarrow K$ is continuous in the weak topology on $L^\infty(U)$. 
\end{proof}

\textbf{Proof of Proposition \ref{Prop: F Lipschitz continuous}}
\begin{proof}
    Consider $\lambda_1, \lambda_2 \in K$ and let $\rho_1$ and $\rho_2$ be the corresponding solutions to \eqref{Eqn: lambda steady state}. Applying Lemma \ref{Lemma: Bounding lambda in L infinity with the norm of rho in Sob} to $\rho_1 - \rho_2$ we have that 
    \begin{align*}
        \big\| \mathcal{F}(\lambda_1) - \mathcal{F}(\lambda_2) \big\|_{L^\infty(U)} 
        &\leq c_2 \sqrt{A} \varM_{\max} \, \| \rho_1 - \rho_2 \|_{W_2^{2,1}(U_A)} \,.
    \end{align*} 
    We wish to bound the terms on the right hand side so that we may compare this directly to $\big\| \lambda_1 - \lambda_2 \big\|_{L^\infty(U)}$.
    
    Following the same approach as in the proof of Proposition \ref{Prop: F weakly continuous}, the equivalent of \eqref{Eqn: bound on q^n} gives 
    \begin{align} \label{Eqn: bound on rho1 - rho2}
        \| \rho_1 - \rho_2 \|_{W_2^{2,1}(U_A)} \leq c_3 \Big( \| f \|_2 + \| \Phi \|_{W_2^{3/2,3/4}(S_A)} + A^{-3/4} \| \Phi \|_{L^2(S_A)} \Big) 
    \end{align}
    where
    \begin{align*}
        f(a,x) &= -\partial_x \bigg( \rho_1(a,x) \Big( \Lambda_1(x) - \Lambda_2(x) \Big) \bigg) \,, \\
        \Phi(a,x) &= -\rho_1(a,x) \Big( \Lambda_1(x) - \Lambda_2(x) \Big) \,.
    \end{align*}
    Recall from the proof of Proposition \ref{Prop: F weakly continuous} that each of the norms appearing in the right hand side of \eqref{Eqn: bound on rho1 - rho2} can be bounded by a constant multiplied by 
    \begin{align*}
        \max\{1,\tau^{-1}\}\,\| \rho_1 \|_{W^{2,1}_2(U_A)} \bigg( 
        \| \Lambda_1 - \Lambda_2 \|_\infty + 
        \| (\Lambda_1)' - (\Lambda_2)' \|_\infty + 
        \| (\Lambda_1)'' - (\Lambda_2)'' \|_\infty \bigg) \,.
    \end{align*}
    As $|\varphi|$ is bounded above by 2 we have
    \begin{align*}
        \big\| \Lambda_1 - \Lambda_2 \big\|_\infty  
        &= \bigg| \int_U \varphi(y-x)\,\lambda_1(y)\,dy - \int_U \varphi(y-x)\,\lambda_2(y)\,dy \bigg| \,,\\
        &\leq \int_U |\varphi(y-x)|\,|\lambda_1(y) - \lambda_2(y)|\,dy \,,\\
        &\leq 2 \int_U |\lambda_1(y) - \lambda_2(y)| \,dy \,,\\
        &\leq 4 \, \big\|\lambda_1 - \lambda_2\big\|_{L^\infty(U)} \,.
    \end{align*}
    Similarly $\big\| \Lambda_1' - \Lambda_2' \big\|_\infty \leq c_6 \, \big\|\lambda_1 - \lambda_2\big\|_{L^\infty(U)} $ and $\big\| \Lambda_1'' - \Lambda_2'' \big\|_\infty \leq c_7 \, \big\|\lambda_1 - \lambda_2\big\|_{L^\infty(U)} $ for some constants $c_6,c_7$ depending only upon $U$ and $\phi$ (and its derivatives). 
        
    As the bound \eqref{Eqn: bound on rho in Sob space} given by 
    \begin{align*}
        \|\rho\|_{W_2^{2,1}(U_A)} \leq c_1 \, \| \mu \|_{L^2(U)} \,,
    \end{align*}
    is independent of $\lambda$ we can therefore bound each of the norms appearing in the right hand side of \eqref{Eqn: bound on rho1 - rho2} by a constant multiplied by 
    \[
    \| \mu \|_{L^2(U)} \, \big\|\lambda_1 - \lambda_2\big\|_{L^\infty(U)} \,.
    \]
    Therefore there exists a constant $c_8$, depending upon $U$, $\phi$, $A$ and $\tau$, such that 
    \begin{align*} 
        \big\| \mathcal{F}(\lambda_1) - \mathcal{F}(\lambda_2) \big\|_{L^\infty(U)} 
        \leq  c_8 \sqrt{A} \varM_{\max} \, \big(1+ A^{-3/4} \big)\, \big\|\lambda_1 - \lambda_2\big\|_{L^\infty(U)} \,.
    \end{align*}
    Hence we see that $\mathcal{F}$ is Lipschitz continuous and satisfies \eqref{Eqn: Lipschitz continuity of F}. 
\end{proof}

\textbf{Proof of Proposition \ref{Prop: Correspondence with MK}}
\begin{proof}
    From \eqref{Eqn: construction of MK solution} we have the following 
    \begin{align*}
        \partial_t \rho(t,a,x) 
        &= \partial_t q(t,a,x) \, \pi(a) \\
        &= -\Big( \tau \partial_a q(t,a,x) + \partial_x F[q](t,a,x) \Big) \, \pi(a) \\
        \partial_a \rho(t,a,x) 
        &= \partial_a q(t,a,x) \, \pi(a)  + q(t,a,x) \, \partial_a \pi(a) \\
        &= \partial_a q(t,a,x) \, \pi(a)  - q(t,a,x) \, d(a) \, \pi(a) \\
        &= \Big( \partial_a q(t,a,x)  - q(t,a,x) \, d(a) \Big) \, \pi(a) \\
        \partial_x \rho(t,a,x) 
        &= \partial_x q(t,a,x) \, \pi(a) 
    \end{align*}
    Next we compare the flux term in \eqref{Eqn: PDE with MK ageing} with that in \eqref{Eqn: PDE with ageing}, 
    \begin{align*}
        & \Tilde{F}[\rho](t,a,x) \\
        \quad& = \rho(t,a,x) \int_U \varphi(y - x) \, \bigg( \int_0^\infty \varM(a,b) \, \rho(t,b,y) \, db \bigg) \, dy - \, \frac{\sigma^2}{2} \partial_{x}\rho(t,a,x) \,, \\
        & = q(t,a,x) \, \pi(a) \int_U \varphi(y - x) \, \bigg( \int_0^\infty \varM(a,b) \, q(t,b,y) \, \pi(b) \, db \bigg) \, dy - \, \frac{\sigma^2}{2} \partial_{x} q(t,a,x) \, \pi(a) \\
        & = \pi(a) \Bigg( q(t,a,x) \, \int_U \varphi(y - x) \, \bigg( \int_0^{1} \varM(a,b) \, q(t,b,y) \, \pi(b) \, db \bigg) \, dy - \, \frac{\sigma^2}{2} \partial_{x} q(t,a,x) \Bigg) \\
        & = \pi(a) \Bigg( q(t,a,x) \int_U \varphi(y - x) \, \bigg( \int_0^1 \varM(a,b) \, \pi(b) \, q(t,b,y) \, db \bigg) \, dy - \, \frac{\sigma^2}{2} \partial_{x} q(t,a,x) \Bigg) \\
        & = \pi(a) \Bigg( q(t,a,x) \int_U \varphi(y - x) \, \bigg( \int_0^1 \varM'(a,b) \, q(t,b,y) \, db \bigg) \, dy - \, \frac{\sigma^2}{2} \partial_{x} q(t,a,x) \Bigg) \\
        & = \pi(a) \, F[q](t,a,x) \,.
    \end{align*}
    Combining this with the derivatives above gives 
    \begin{align*}
        \partial_t \rho(t,a,x) + \tau \partial_a \rho(t,a,x) +  \, \partial_x \Tilde{F}[\rho](t,a,x) 
        &= \pi(a) \bigg( -\tau \partial_a q(t,a,x) - \partial_x F[q](t,a,x) \\
        &\hspace{1.5cm} + \tau \partial_a q(t,a,x) - \tau \, q(t,a,x) \, d(a) \\[0.4em]
        &\hspace{1.5cm} +  \, \partial_x F[q](t,a,x) \bigg) \\[0.2em]
        &= - \tau \, p(t,a,x) \, d(a) 
    \end{align*}
    hence we have 
    \begin{align*}
        \partial_t \rho(t,a,x) + \partial_a \rho(t,a,x) +  \, \partial_x \Tilde{F}[\rho](t,a,x) = -\tau \, \rho(t,a,x) \, d(a) 
    \end{align*}
    as required. In addition the correspondence $\Tilde{F}[\rho](t,a,x) = \pi(a) \, F[q](t,a/A,x) $ immediately ensures that the no-flux boundary conditions at $x=\pm1$ are met. Furthermore both equations share the same initial ($t=0$) condition. The last condition to check is the age-zero distribution. From \eqref{Eqn: construction of MK solution} we have
    \begin{align*}
        \rho(t,0,x) &= q(t,0,x) \, \pi(0) \\
        &= \mu(x) \, \pi(0) \, \int_U q(t,1,y) \, dy \\
        &= \mu(x) \, \pi(0) \,,
    \end{align*}
    since the age profile for $q$ is uniform. Since $\pi$ is stationary it maintains a constant population size and so satisfies
    \begin{align*}
        \pi(0) = \int_0^\infty \vard(a) \, \pi(a)  \, da \,,
    \end{align*}
    which gives
    \begin{align*}
        \rho(t,0,x) = \mu(x) \bigg( \int_0^\infty \vard(a) \, \pi(a) \, da \bigg) \,,
    \end{align*}
    as required. 
\end{proof}

\section{Numerical Schemes}  \label{Section: Numerical Schemes} 

\subsection{Full System} \label{Appendix: Numerics}

The overall approach is to appply Strang Splitting \citep{holden2010splitting} to allow separate numerical schemes for ageing and opinion formation. At each timestep we apply a half-step in age transport, a full step in opinion formation, then a further half-step in age transport. For ageing we solve
\begin{align*}
    \partial_t\rho + \tau \partial_a\rho  = 0 \,,
\end{align*}
while for opinion formation we solve
\begin{align*}
    \partial_t\rho + \partial_x \Bigg( \rho(t,a,x) \bigg( \int_U \int_0^A \varM(a,b) \, \varphi(y - x) \, \rho(t,b,y) \, db \, dy \, \bigg) \Bigg) - \, \frac{\sigma^2}{2} \partial_{x^2} \rho = 0 \,.
\end{align*}

In the following we assume $M\equiv1$ for simplicity, but note where an adaptation would be made to account for more complex age interaction kernels. 

Fix $J_t,J_x,J_a\geq 1$ the number of discretisation points for time, opinion and age respectively. 

Let $\Delta t = J_t^{-1}, \, \Delta x = 2J_x^{-1}$ and $\Delta a = J_a^{-1}$. Denote $x_i = i \, \Delta_x$ for $i=0,\dots,J_x$. 

\textbf{Half age step:}

Denote by $\mu^{J_x}$ the discretised version of the age zero distribution $\mu$ into $J_x$ opinion sections, given by 
\begin{align*}
    \mu^{J_x}_i = \int_{x_{i-1}}^{x_i} \mu(x) \, dx \,,
\end{align*}
for $i=1,\dots,J_x$.

Let $\kappa = \frac{\tau}{2}\frac{\Delta t}{\Delta a}$ and define the matrices $B^{1}\in\mathds{R}^{J_a \times J_a}$ and $B^{2}\in\mathds{R}^{J_x \times J_a}$ by 
\begin{align*}
    B^{1} = 
    \begin{pmatrix}
        0 & \kappa & 0 & 0 & 0 & 0\\
        0 & 1-\kappa & \kappa & \dots & 0 & 0\\
        0 & 0 & 0 & \dots & 0 & 0\\
        \vdots & \vdots & \vdots & \ddots & \vdots & \vdots \\
        0 & 0 & 0 & \dots & 1-\kappa & \kappa \\
        0 & 0 & 0 & \dots & 0 & 1-\kappa
    \end{pmatrix} 
    \quad \text{and} \quad
    B^{2} = 
    \begin{pmatrix}
        \mu^{J_x}_1 & 0 & 0 & 0 & 0 & 0\\
        \mu^{J_x}_2 & 0 & 0 & \dots & 0 & 0\\
        \mu^{J_x}_3 & 0 & 0 & \dots & 0 & 0\\
        \vdots & \vdots & \vdots & \ddots & \vdots & \vdots \\
        \mu^{J_x}_{{J_x}-1} & 0 & 0 & \dots & 0 & 0 \\
        \mu^{J_x}_{J_x} & 0 & 0 & \dots & 0 & 0
    \end{pmatrix}  \,,
\end{align*}
then we have that a half update step corresponds to 
\begin{equation*}
    S^a(\rho) = \rho B^{1} + \big( \mathbf{1}_{1\times {J_x}} \, \rho \, e_{J_a} \big) B^{2}
\end{equation*}
where $\mathbf{1}_{n\times m}$ is an $n\times m$ matrix of ones. The matrix $B^{1}$ handles the transport in age for ages strictly above zero, while $B^{2}$ sets the age-zero distribution. The factor $\big( \mathbf{1}_{1\times {J_x}} \, \rho \, e_{J_a} \big)$ simply sums the age-one distribution to ensure the total mass of $\rho$ is preserved. 

\textbf{Opinion Step:}

The opinion step applied a finite volumes scheme \cite{leveque2002finite}. Define the constant matrix $\Phi\in\mathds{R}^{({J_x}-1) \times {J_x}}$, which describes opinion interactions, by
\begin{equation*}
    \Phi_{ij} = \int_{x_{j-1}}^{x_j} \phi(y - x_i) \, (y - x_i) \, dy \,.
\end{equation*}
for $i=1,\dots,{J_x}-1$ and $j=1,\dots,{J_x}$. 

To help approximate the fluxes define $F^{1}, F^{2}\in\mathds{R}^{({J_x}-1) \times {J_x}}$ as follows 
\begin{align*}
    F^{1} = -\frac{1}{2}
    \begin{pmatrix}
        1 & 1 & 0 & \dots & 0 & 0\\
        0 & 1 & 1 & \dots & 0 & 0\\
        \vdots & \vdots & \vdots & \ddots & \vdots & \vdots \\
        0 & 0 & 0 & \dots & 1 & 0 \\
        0 & 0 & 0 & \dots & 1 & 1
    \end{pmatrix}
    \quad , \quad
    F^{2} = \frac{1}{\Delta x}
    \begin{pmatrix}
        -1 & 1 & 0 & \dots & 0 & 0\\
        0 & -1 & 1 & \dots & 0 & 0\\
        \vdots & \vdots & \vdots & \ddots & \vdots & \vdots \\
        0 & 0 & 0 & \dots & 1 & 0 \\
        0 & 0 & 0 & \dots & -1 & 1
    \end{pmatrix} \,.
\end{align*}
The matrix $F^{1}$ approximates the value of $\rho$ in the midpoint of each section (over which finite volumes are taken), while $F^{2}$ approximates the spatial derivative appearing in the flux. 

We now define the matrix $C\in\mathds{R}^{{J_x} \times ({J_x}-1)}$ that combines the fluxes (accounting for the no-flux boundary conditions)
\begin{align*}
    C = 
    \begin{pmatrix}
        1 & 0 & 0 & \dots & 0 & 0\\
        -1 & 1 & 0 & \dots & 0 & 0\\
        \vdots & \vdots & \vdots & \ddots & \vdots & \vdots \\
        0 & 0 & 0 & \dots & -1 & 1 \\
        0 & 0 & 0 & \dots & 0 & -1
    \end{pmatrix} \,.
\end{align*}

An opinion update step then corresponds to 
\begin{equation} \label{Eqn: Numerical opinion step}
    S^o(\rho) = \rho + \frac{\Delta t}{\Delta x} C \, \bigg( \Big(F^{1} \rho\Big)\Big( \Phi \rho \mathbf{1}_{M\times M}\Big) \,+\, \frac{\sigma^2}{2} F^{2} \rho \bigg) \,.
\end{equation}
If $\varM$ is not always $1$ then $\mathbf{1}_{{J_a}\times {J_a}}$ would be replaced with an approximation of $\varM$. 

\textbf{Complete Step}

A full step is given by 
\begin{equation*}
    \rho(t+\Delta t) = \Big( S^a \circ S^o \circ S^a \Big) \big( \rho(t) \big) \,,
\end{equation*}
where $\circ$ denotes function composition.

\subsection{For steady states} \label{Appendix: Numerics for steady states}

We now adapt the scheme for a known interaction density $\lambda$. Recall that in \eqref{Eqn: lambda steady state}, age plays the role of `time'. 

As for $\mu$, denote by $\lambda^{J_x}$ the discretised version of $\lambda$ into $J_x$ opinion sections, given by 
\begin{align*}
    \lambda^{J_x}_i = \int_{x_{i-1}}^{x_i} \lambda(x) \, dx \,,
\end{align*}
for $i=1,\dots,J_x$. We then replace \eqref{Eqn: Numerical opinion step} with
\begin{equation} 
    S^o(\rho) = \rho + \frac{\Delta a}{\Delta x} C \, \bigg( \Big(F^{1} \rho\Big)\star \Big( \Phi \lambda^{J_x} \Big) \,+\, \frac{\sigma^2}{2} F^{2} \rho \bigg) \,.
\end{equation}
where $\star$ denotes elementwise multiplication. 

This can be simplified by defining
\begin{align*}
    \Lambda^{J_x} = \text{diag}\Big( \Phi \lambda^{J_x} \Big) \,,
\end{align*}
which gives
\begin{align*}
    S^o(\rho) 
    &= \rho + \frac{\Delta a}{\tau\Delta x} C \, \bigg( \Big(\Lambda^{J_x} F^{1} \rho\Big)\,+\, \frac{\sigma^2}{2} F^{2} \rho \bigg) \,, \\
    &= \rho + \frac{\Delta a}{\tau\Delta x} C \, \bigg( \Lambda^{J_x} F^{1} \rho \,+\, \frac{\sigma^2}{2} F^{2} \rho \bigg) \,, \\
    &= \rho + \frac{\Delta a}{\tau\Delta x} C \, \bigg( \Lambda^{J_x} F^{1} \,+\, \frac{\sigma^2}{2} F^{2} \bigg) \,\rho \,, \\
    &= \Bigg( \text{Id} + \frac{\Delta a}{\tau\Delta x} C \, \bigg( \Lambda^{J_x} F^{1} \,+\, \frac{\sigma^2}{2} F^{2} \bigg) \Bigg) \,\rho \,.
\end{align*}
where $\text{Id}$ is the ${J_x}\times {J_x}$ identity matrix. Hence the approximation $\rho_k = \rho(k \, \Delta a)$ satisfies 
\begin{align*}
    \rho_{k} = \Bigg( \text{Id} + \frac{\Delta a}{\tau \Delta x} C \, \bigg( \Lambda^{J_x} F^{1} \,+\, \frac{\sigma^2}{2} F^{2} \bigg) \Bigg)^k \, \mu^{J_x} \,.
\end{align*}
For ease of notation define 
\begin{align*}
    \Omega = \frac{1}{\tau \Delta x} C \, \bigg( \Lambda^{J_x} F^{1} \,+\, \frac{\sigma^2}{2} F^{2} \bigg)
\end{align*}
then we have 
\begin{align*}
    \rho_k = \Bigg( \text{Id} + \frac{\Omega}{J_a} \Bigg)^k \mu \approx \Bigg( \text{Id} + \frac{\Omega}{J_a} \Bigg)^{a\, J_a} \mu \approx e^{a\Omega} \mu
\end{align*}
for $a\in[0,1]$. 

This expression can be used to rapidly solve \eqref{Eqn: lambda steady state} and therefore iterate a numerical approximation of $\mathcal{F}$ to find steady states. Due to the simpler nature of \eqref{Eqn: lambda steady state} and the corresponding numerical scheme, this is significantly faster and more stable than solving the PDE for a large $T$. In addition, since the PDE may exhibit periodic solutions, solving for large $T$ may never yield a steady state. 

\end{document}